\documentclass[11pt]{amsart}
\usepackage{amsmath,amssymb,amscd,amsthm}
\usepackage{amsthm}
\usepackage[margin=1.3in]{geometry}
\usepackage[pagewise]{lineno}

\usepackage{tikz-cd}
\usetikzlibrary{babel}
\usetikzlibrary{decorations.pathmorphing}
\usepackage[margin=1.3in]{geometry}
\usepackage{hyperref}

\usepackage{hyperref}

\hypersetup{hidelinks, colorlinks=true, allcolors=blue, pdfstartview=Fit, breaklinks=true}
\usepackage[all]{xy}

\usepackage{enumerate}
\usepackage{paralist}
\usepackage{amscd}
\usepackage{extarrows}

\newtheorem{thm}{Theorem}[subsection]
\newtheorem{lem}[thm]{Lemma}
\newtheorem{prop}[thm]{Proposition}

\newtheorem{cor}[thm]{Corollary}

\newtheorem*{prob}{\bf Problem}
\theoremstyle{definition}\newtheorem{df}[thm]{Definition}
\theoremstyle{definition}\newtheorem{rem}[thm]{Remark}
\newtheorem{exm}[thm]{\it Example}
\newtheorem{qn}{\bf Question}

\theoremstyle{definition}

\renewcommand{\phi}{\varphi}

\newcommand{\N}{\mathbb{N}}
\newcommand{\Homeo}{\mathcal{H}}

\newcommand{\A}{\mathcal{A}}
\newcommand{\Z}{\mathbb{Z}}
\newcommand{\J}{\mathcal{J}}
\newcommand{\tildeX}{\widetilde{X}}
\newcommand{\X}{\underline{X}}

\newcommand{\KX}{{}_{k_2}X_{l_2}}
\newcommand{\kX}{{}_{k_1}X_{l_1}}
\newcommand{\Kx}{{}_{k_2}[x]_{l_2}}
\newcommand{\kx}{{}_{k_1}[x]_{l_1}}
\newcommand{\jf}{\mathfrak{j}}
\newcommand{\I}{\mathcal{I}}
\newcommand{\D}{\mathcal{D}}
\newcommand{\Sph}{\mathbb{S}}
\newcommand{\La}{\mathcal{L}}
\newcommand{\R}{\mathbb{R}}
\newcommand{\Raf}{\mathcal{R}_0(\tilde{\alpha})}
\newcommand{\Rbt}{\mathcal{R}_0(\tilde{\beta})}
\newcommand{\C}{\mathbb{C}}
\newcommand{\T}{\mathbb{T}}
\newcommand{\taf}{\tilde{\alpha}}
\newcommand{\tbt}{\tilde{\beta}}

\newcommand{\Aff}{\operatorname{Aff}}
\newcommand{\Pj}{\mathcal{P}}

\newcommand{\id}{\operatorname{id}}

\newcommand{\ep}{\varepsilon}
\newcommand{\F}{{\mathcal F}}

\newcommand{\Lip}{\mathcal{L}}
\newcommand{\G}{\mathcal{G}}

\newcommand{\af}{{\alpha}}
\newcommand{\bt}{{\beta}}

\newcommand{\beq}{\begin{eqnarray}}
\newcommand{\eneq}{\end{eqnarray}}

\begin{document}
\title[Dim. associated with surjective local homeo. and subshifts with low complexity]{Dimensions associated with surjective local homeomorphisms and subshifts with low complexity}

\author{Zhuofeng He, Sihan Wei}

\maketitle

\begin{abstract} 
We prove that the Cuntz-Pimsner algebra associated to any surjective aperiodic one-sided subshift with finitely many left special elements has finite nuclear dimension, which is especially the case for every surjective aperiodic subshift with nonsuperlinear-growth complexity.

As a generalization, we define the notions of left speical set, the topological Rokhlin dimension, the tower dimension and the amenability dimension for every local homeomorphism. Then we turn to prove that, for every surjective local homeomorphism with a finite left special set consisting of isolated points, these dimensions along with the dynamic asymptotic dimension are all finite.

\quad\par

\end{abstract}

\section{introduction}\label{sec:1}
For every one-sided subshift $X$ over a finite alphabet $\mathcal{A}$, a {\it left special element} is a point $x\in X$ such that there are at least two distinct letters $\af,\bt\in\mathcal{A}$ with $\af x,\bt x\in X$. In \cite{HW}, the authors show that for every (nontrivial) minimal one-sided subshift with finitely many left special elements, the nuclear dimension of its Cuntz-Pimsner  $C^*$-algebra is finite and in particular, is always 1. Recall that for any subshift $X$, its Cuntz-Pimsner algebra is defined by K. Brix and T. Carlsen in \cite{BC} to be the full groupoid $C^*$-algebra $C^*(\mathcal{G}_{\tilde{X}})$, where $\mathcal{G}_{\tildeX}$ is the topological groupoid associated with $\tildeX$, the so-called “cover” of $X$, which is a zero-dimensional compact metric space with a surjective local homeomorphism on it. Our approach was, generalizing K. Brix's method in \cite{B}, describing the cover system $\tildeX$, showing that it is the union of a minimal two-sided subshift and finitely many discrete orbits. 

On the other hand, passing from the one-sided subshift to the classification program of $C^*$-algebras raised and developed by G. Elliott and his countless collaborators, several notions have been defined during the great efforts dealing with crossed products induced by group actions on compact metric spaces, especially those dimensional ones. These notions include {\it tower dimension} and {\it fine tower dimension}  of D. Kerr and G. Szab$\acute{\rm o}$ in \cite{KS}, {\it dynamic asymptotic dimension} and {\it amenability dimension} of E. Guentner, R. Willett and G. Yu in \cite{GWY}, and {\it topological Rokhlin dimension} of $\mathbb{Z}^k$-actions defined by G. Szab$\acute{\rm o}$ in \cite{Sza}, to name a few.

Combining these two facets, we have the following two questions, which are also the motivations of the article.

\begin{qn}
What if $X$ is an aperiodic one-sided subshift rather than a minimal one? Does the Cuntz-Pimsner algebra $\mathcal{O}_X$ still have finite nuclear dimension, given that the number of its left-special elements is finite?
\end{qn}

\begin{qn}
Can we define those aforementioned dimensional notions, in particular, the topological Rokhlin dimension, amenability dimension and tower dimension for surjective local homeomorphisms on compact metric spaces?
\end{qn}

As for Question 1, we observe in subsection 3.1 that, if $X$ is an aperiodic one-sided subshift with finitely many left special elements, then the number of the points in $(\tildeX,\sigma_{\tildeX})$ having at least two preimages under $\sigma_{\tildeX}$ is exactly equal to the number of left special elements in $X$, see Theorem \ref{shift}. Besides, each of these points is isolated in $\tildeX$. This enlightens us to define the notion of {\it left special set} for a surjective local homeomorphism $\phi$ on a compact metric space $X$ to be the points in $X$ having at lease two preimages under $\phi$. Then with this observation, we conclude in Corollary \ref{cor332} that the Cuntz-Pimsner algebra of every aperiodic one-sided subshift with finitely many left special elements has finite nuclear dimension, and this is especially the case for every aperiodic subshift $X$ with nonsuperlinear-growth complexity, since we also show in Lemma \ref{useful} that every such subshift has finitely many left special elements.

Then for Question 2, we define the notion of the {\it topological Rokhlin dimension}(see subsection 2.9), the {\it tower dimension} (Definition \ref{td}) and the {\it amenability dimension} (Definition \ref{amd}) for a local homeomorphism. And with our notion of left special set $Sp_l(X,T)$ associated to every (unnecessarily invertible) dynamical system $(X,T)$, we show the following implications in Section \ref{sec:4} and Section \ref{sec:5}:

(i) Every aperiodic surjective local homeomorphism on a zero-dimensional compact metric space with $|Sp_l(X,T)|<\infty$ has finite topological Rokhlin dimension, with the inequality ${\rm dim}_{\rm Rok}(X,T)\le 2|Sp_l(X,T)|+1$(Theorem \ref{zdr});

(ii) Every local homeomorphism $T$ on a compact metric space $X$ satisfies the inequality ${\rm dim}_{\rm tow}(X,T)\le 2{\rm dim}_{\rm Rok}(X,T)+1$(Lemma \ref{td1});

(iii) Every surjective local homeomorphism $T$ on a compact metric space $X$ with $|Sp_l(X,T)|<\infty$ fits into the inequality ${\rm dim}_{\rm am}(X,T)\le {\rm dim}_{\rm tow}(X,T)$, whenever every element of $Sp_l(X,T)$ is an isolated point in $X$(Lemma \ref{fad});

(iv) Every surjective local homeomorphism $T$ on a compact metric space $X$ with $|Sp_l(X,T)|<\infty$ makes the inequality ${\rm dad}(\mathcal{G}_{(X,T)})\le{\rm dim}_{\rm am}(X,T)$ hold, whenever every element of $Sp_l(X,T)$ is an isolated point in $X$(Lemma \ref{fdad}).

As a result, for every aperiodic surjective local homeomorphism on a zero-dimensional compact metric space with finitely many left special elements, if every its left special element is an isolated point, then all of these dimensions are finite.

Finally, we shall also mention that, in the definition of these dimensional notions, we almost have to replace every image $T^n(Y)$ with the preimage $T^{-n}(Y)$, due to the simple fact of a local homeomorphism: $Y_1\cap Y_2=\varnothing$ does NOT imply that $T(Y_1)\cap T(Y_2)=\varnothing$!

We also ask the following question.
\begin{qn}
Do these dimensions, especially the topological Rokhlin dimension, really have anything to do with the number of left special elements?
\end{qn}

\subsection{Outline of the paper}
The paper is organized as follows. 

Section \ref{sec:2} will provide definitions, including basic notions of local homeomorphisms, aperiodicity and orbits. The definition of left special elements and the topological Rokhlin dimension for a local homeomorphism will be given as well. Moreover, we will also explain what does it mean by a subshift with nonsuperlinear-growth complexity, and show that every such aperiodic subshift has finitely many left special elements(Lemma \ref{useful}).

In Section \ref{sec:3}, we prove that every surjective aperiodic one-sided subshift with finitely many left special elements has finite nuclear dimension(Corollary \ref{cor332}). This then applies to any surjective aperiodic one-sided subshift with nonsuperlinear-growth complexity.

Section \ref{sec:4} is devoted to the finiteness of topological Rokhlin dimension for a surjective local homeomorphism on a zero-dimensional compact metric space with finitely many left special elements(Theorem \ref{zdr}). Several technical lemmas are given as preparations.

We finally define and consider in Section \ref{sec:5} the finiteness of tower dimension, amenability dimension and dynamic asymptotic dimension for a local homeomorphism on a compact metric space. The inequalities will be given in order, see Lemma \ref{td1}, Lemma \ref{fad} and Lemma \ref{fdad}.

\section{preliminaries}\label{sec:2}
Throughout the whole article, $\N$ is used to denote the set of all nonnegative integers. The compact metric spaces considered will all be nonempty.

Capital Letters $X,Y,\cdots$ are used to denote topological spaces, while small ones $x,y,\cdots$ for points in the underlying spaces.

For a set $A$, we will use the notation $|A|$ to stand for its cardinality. If $A$ is an infinite set, we will simply write $|A|=\infty$.

For a collection $\mathcal{S}$ of subsets of a set $X$, we will write $\bigcup\mathcal{S}$ and $\bigcap\mathcal{S}$ for the union and intersection of elements in $\mathcal{S}$ respectively.

For a metric space $X$, a point $x\in X$ and a positive number $\delta>0$, we write
\[B(x,\delta)=\{y\in X: d(y,x)<\delta\}.\]
If $E\subset X$ is a subset, then we also write
\[B_{\delta}(E)=\{y\in X: d(y,E)<\delta\},\]
where $d(y,E)$ is defined as the $\inf\{d(y,z): z\in E\}$.

We will use $\overline{E}$ and ${\rm int}(E)$ to denote the closure and interior of a subset $E$ of a topological space, respectively.

\subsection{Zero-dimensional spaces}
Let $X$ be a topological space. We say that $X$ is a {\it zero-dimensional space}, if for every $x\in X$ and every open set $x\in U\subset X$, there is a clopen subset $V\subset X$ with $x\in V\subset U$.

\subsection{Local homeomorphisms}
Let $X$ be a topological space, and $T: X\to X$ a continuous map. We say $T$ is a {\it local homeomorphism}, if for every $x\in X$, there is an open set $U$ containing $x$ such that the restriction
\[T|_U: U\to T(U)\]
is a homeomorphism. It is well known that

(A) Every local homeomorphism is an {\it open map}, that is, sending open sets to open sets;

(B) Every fibre of a local homeomorphism is a discrete subspace of its domain, where by a {\it fibre}, we mean  the inverse images $T^{-1}(\{y\})$ for $y\in X$;

(C) A map is a local homeomorphism if and only if it is continuous, open and {\it locally injective};

(D) The composition of two local homeomorphisms is a local homeomorphism as well.

Throughout the paper, we will use these facts once needed, without any further proof.

\subsection{Aperiodic local homeomorphisms and free dynamical systems}
Let $X$ be a compact metric space and $T: X\to X$ a local homeomorphism. If there are $n\in\N$ and $x\in X$ such that $T^n(x)=x$, we then say $x$ is a {\it periodic point} of $T$. The set of periodic points of $T$ in $X$ is denoted by $P(X,T)$.

If $P(X,T)=\varnothing$, we shall say $T$ is an {\it aperiodic local homeomorphism}, and $(X,T)$ a {\it free dynamical system}.

\subsection{The orbits of a dynamical system}
Let $X$ be a compact metric space and $T: X\to X$ a local homeomorphism. 

For every $x\in X$, we define its {\it forward orbit} and {\it backward orbit} to be
\[{\rm Orb}^+_T(x)=\{T^n(x): n\ge0\}\ \ {\rm and}\ \ {\rm Orb}^-_T(x)=\{y\in X: T^n(y)=x\ {\rm for\ some\ }n>0\}\]
respectively, and its {\it (whole) orbit}
\[{\rm Orb}_T(x)={\rm Orb}^+_T(x)\cup{\rm Orb}^-_T(x).\]

\subsection{The Lebesgue covering dimension and small inductive dimension}
Let $X$ be a topological space and $\mathcal{U}$ is a finite open cover of $X$. Define the {\it order} ${\rm ord}(\mathcal{U})$ of $\af$ by
\[{\rm ord}(\mathcal{U})=-1+\max_{x\in X}\sum_{O\in\mathcal{U}}1_O(x).\]
Let 
\[D(\mathcal{U})=\min_{\mathcal{U}'\ {\rm refines}\ \mathcal{U}}{\rm ord}(\bt),\]
where by $\mathcal{U}'$ {\it refines} $\mathcal{U}$, we mean that $\mathcal{U}'$ is a finite open cover satisfying that every element of $\mathcal{U}'$ is contained in some element of $\mathcal{U}$. The {\it Lebesgue covering dimension} is given by
\[{\rm dim}(X)=\sup_{\mathcal{U}}D(\mathcal{U}),\]
where $\mathcal{U}$ runs over all finite open covers of $X$. It is well known that every compact metric space of dimension $d$ embeds into $\R^{2d+1}$.

For the {\it small inductive dimension} ${\rm ind}(X)$, we start with defining
\[{\rm ind}(\varnothing)=-1.\]
Then inductively, the {\it small inductive dimension} ${\rm ind}(X)$ of $X$ is the smallest $n$ such that for every $x\in X$ and every open set $U$ containing $x$, there is an open set $V$ with $x\in V\subset\overline{V}\subset U$ and ${\rm ind}(\partial V)\le n-1$.

The following properties in dimension theory hold for every separable metric space $X$(and especially for every compact metric space). Readers may refer to \cite{Eng} for their proofs.

(D1) ${\rm ind}(X)={\rm dim}(X)$;

(D2) (The Countable Closed Sum Theory) For a sequence of closed subsets $\{B_i\}_{i\ge1}$ with ${\rm dim}(B_i)\le k$ for every $i\ge1$, we have ${\rm dim}(\bigcup_{i\ge1}B_i)\le k$;

(D3) Let $E$ be a zero-dimensional subset of $X$. Then for every $x\in X$, and every open set $U$ containing $x$, there is an open set $U'$ with $x\in U'\subset U$ and $E\cap\partial U'=\varnothing$;

(D4) If $X\ne\varnothing$, then there is a zero-dimensional set $E\subset X$ which is $F_\sigma$ in $X$ and such that ${\rm dim}(X\setminus E)={\rm dim}(X)-1$.

Since the underlying spaces throughout the paper will all be compact metric spaces(or their subspaces, and hence are separable metric spaces), according to (D1), we will hence not distinguish Lebesgue covering dimension and small inductive dimension.


\subsection{The general position}
Let $X$ be a compact metric space with finite dimension $d$. Let $B$ be a family of subsets of $X$. We say that $B$ is in {\it general position} if for every finite subfamily $S\subset B$ with $m$ elements, 
\[{\rm dim}(\cap S)\le \max\{-1, d-m\}.\]
Note that if ${\rm dim}(X)=0$, then $B$ is in general position if and only if 
\[B=\{\varnothing\}.\]
Readers who are interested in this property may refer to \cite{Kul} for details.





\subsection{Left special elements}
Let $X$ be a compact metric space, $x\in X$ and $T:X\to X$ a continuous map. 

We say that $x$ is a {\it left special element} of the topological dynamical system $(X,T)$, if
\[|T^{-1}(\{x\})|\ge2.\]
The set of left special elements of $(X,T)$ is denoted by $Sp_l(X,T)$. We shall call $Sp_l(X,T)$ the {\it left special set} of $(X,T)$.

\subsection{The Rokhlin covers of local homeomorphisms}
Let $(X,T)$ be a topological dynamical system, where $X$ is a compact metric space and $T:X\to X$ is a local homeomorphism. 

Let $N\in\N$. By an {\it $N$-Rokhlin tower} in of $(X,T)$, we mean the following collection of nonempty sets
\[\mathcal{T}(U_0,N)=\{U_0, U_1,\cdots, U_{N-1}\}\]
in $X$ such that 

(1) $U_{k}=T^{-1}(U_{k-1})$ for $k=1,2,\cdots,N-1$, and,

(2) $\overline{U_i}\cap\overline{U_j}=\varnothing$ for distinct $0\le i,j\le N-1$.

We shall say an $N$-Rokhlin tower is {\it open} if each of its members is open, and {\it clopen} if each of its members is clopen. 

An $N$-{\it Rokhlin cover} of $(X,T)$ is a {\it finite} collection of open $N$-Rokhlin towers
\[\mathcal{R}=\bigcup_{0\le l\le d}\mathcal{T}(U_0^{l}, N)\]
such that $X=\bigcup\mathcal{R}$, i.e., $\mathcal{R}$ forms an open cover of $X$.

\subsection{The topological Rokhlin dimension of $(X,T)$}
Let $(X,T)$ be a topological dynamical system, where $X$ is a compact metric space and $T:X\to X$ is a local homeomorphism. 

Let $d\in\N$. We say that $(X,T)$ has the {\it topological Rokhlin dimension} no more than $d$, denoted by ${\rm dim}_{\rm Rok}(X,T)\le d$, if for every $N\ge1$, $(X,T)$ admits an $N$-Rokhlin cover $\mathcal{R}$, whose number of $N$-Rokhlin towers is less or equal to $d+1$.

If ${\rm dim}_{\rm Rok}(X,T)\le d$ and ${\rm dim}_{\rm Rok}(X,T)\nleq d-1$, then we say the topological Rokhlin dimension of $(X,T)$ is $d$, and write ${\rm dim}_{\rm Rok}(X,T)=d$. If there is no such $d$, we then write ${\rm dim}_{\rm Rok}(X,T)=\infty$.

\subsection{The groupoid of a local homeomorphism}
Let $X$ be a compact metric space and $T:X\to X$ a local homeomorphism. We then obtain a topological dynamical system $(X, T)$.  The corresponding {\it Deaconu-Renault Groupoid} is defined to be the set
\[\mathcal{G}_{(X,T)}=\{(x,m-n,y)\in X\times \Z\times X: T^m(x)=T^n(y), m,n\in\N\},\]
with the unit space $\mathcal{G}_{(X,T)}^{(0)}=\{(x,0,x): x\in X\}$ identified with $X$, range and source maps defined as $r(x,n,y)=(x,0,x)$ and $s(x,n,y)=(y,0,y)$, and operations $(x,n,y)(y,m,z)=(x,n+m,z)$ and $(x,n,y)^{-1}=(y,-n,x)$.

\subsection{The subshifts over a finite alphabet}
Let $\A$ be a finite {\it alphabet}, i.e., a finite set endowed with the discrete topology, whose elements are referred to as the {\it letters}. With the product topology on $\A^\N$, a {\it one-sided subshift} is a closed subspace $X\subset\A^\N$ invariant under the continuous map
\[\sigma:\A^\N\to\A^\N, (x_n)_{n\ge0}\mapsto(x_{n+1})_{n\ge0}.\]
Here by {\it invariant}, we mean $\sigma(X)\subset X$. As we will {\it only} consider one-sided subshifts, the term “one-sided” will be omitted, unless otherwise noted.

\subsection{The complexity function of a subshift}
Let $X$ be a subshift over a finite alphabet $\A$. The {\it language} $\La(X)$ of $X$ is the set of {\it words} over $\A$ of finite length that occur in at least one $x\in X$. Then $\La(X)\cap\A^n$ is a finite set, whose elements are exactly those words of length $n$ that occur in some $x\in X$. 

The {\it complexity function} $p_X:\N\setminus\{0\}\to\N\setminus\{0\}$ of $X$ is defined by
\[p_X(n)=|\La(X)\cap\A^n|.\]

\subsection{The linear-growth complexity and nonsuperlinear-growth complexity}

Let $X$ be a subshift. We will say that $X$ has the {\it linear-growth complexity}, if 
\[\limsup_{n\to\infty}p_X(n)/n<\infty,\]
and the {\it nonsuperlinear-growth complexity}, if
\[\liminf_{n\to\infty}p_X(n)/n<\infty.\]
It is clear that every subshift with the linear-growth complexity has the nonsuperlinear-growth complexity.

We first introduce a useful lemma, connecting the number of left special elements and the nonsuperlinear-growth complexity.
\begin{lem}\label{useful}
Let $X$ be an aperiodic subshift.  If $X$ has nonsuperlinear-growth complexity, then $X$ has finitely many left special elements. In particular,
\[|Sp_l(X,T)|\le \lceil 2d\rceil\]
where $d=\liminf_{n\to\infty}p_X(n)/n$.
\end{lem}

\begin{proof}
Fix any $0<\varepsilon<1/2$. Then there are infinitely many $n$ such that $p_X(n)\le (d+\varepsilon)n$. By Lemma 5 in \cite{Es}, there are infinitely many $m$ such that
\[p_X(m+1)-p_X(m)\le 2(d+\varepsilon).\]
Let us assume that $|Sp_l(X,\sigma)|\ge\lceil 2d\rceil+1$. Then we can choose at least $\lceil 2d\rceil+1$ distinct left special elements of $X$, denote by $\omega_1,\omega_2,\cdots, \omega_{\lceil 2d\rceil+1}$, a positive integer $N$ such that the following prefixes with length $N$
\[(\omega_1)_{[0,N)}, (\omega_2)_{[0,N)},\cdots, (\omega_{\lceil 2d\rceil+1})_{[0,N)}\]
are distinct. By the assumption, we may choose an $m\in\N$ with $m>N$ and $p_X(m+1)-p_X(m)\le 2(d+\varepsilon)$. Note that 
\[(\omega_1)_{[0,m)}, (\omega_2)_{[0,m)},\cdots, (\omega_{\lceil 2d\rceil+1})_{[0,m)}\]
are then $\lceil 2d\rceil+1$ distinct prefixes of length $m$, each of which can be extended to the left in at least two ways, since every $(\omega_i)_{[0,m)}$ is a prefix of some left special element. This immediately follows that
\[p_X(m+1)-p_X(m)\ge \lceil 2d\rceil+1\ge2d+1>2(d+\varepsilon),\]
which contradicts the choice of $m$.
\end{proof}

\begin{rem}
Although the Lemma 5 in \cite{Es} is settled for two-sided shift spaces, one can easily see that the argument there also fit the case of any one-sided shift space.
\end{rem}

\section{The motivation: aperiodic one-sided subshifts}\label{sec:3}

\subsection{The left special set in the cover system}

\begin{df}[\cite{BC}, Definition 2.1]\label{3.4.3}
Let $X$ be a one-sided subshift with $\sigma(X)=X$. By the {\it cover} $\widetilde{X}$ of $X$, we mean the projective limit $\mathop{\lim}\limits_{\longleftarrow}({}_kX_l, {}_{(k,l)}Q_{(k',l')})$. The shift operation $\sigma_{\widetilde{X}}$ on $\widetilde{X}$ is defined so that ${}_k\sigma_{\widetilde{X}}(\tilde{x})_l={}_k[\sigma({}_{k+1}\tilde{x}_l)]_l$ where ${}_{k+1}\tilde{x}_l$ is a representative of the $\stackrel{k+1,l}{\sim}$-equivalence relation class in $\tilde{x}$.
\end{df}
The following sets give a base for the topology of $\tildeX$:
\[U(z,k,l)=\{\tilde{x}\in\tildeX: z\stackrel{k,l}{\sim}{}_k\tilde{x}_l\}\]
for $z\in X$ and $(k,l)\in\I$. It is known that $\sigma_{\tildeX}$ is a surjective local homeomorphism on the zero-dimensional space $\tildeX$, refer to \cite{BC} for details.

\begin{df}[\cite{BC}, 2.1]
Let $\pi: \tildeX\to X$ to be the map which sends each $\tilde{x}\in\tildeX$ to a point $x=\pi(\tilde{x})$ so that $x_{[0,k)}$ are determined uniquely by $({}_k\tilde{x}_l)_{[0,k)}$ for every $(k,l)\in\I$. Define $\imath: X\to\tildeX$ by ${}_k\imath(x)_l={}_k[x]_l$ for every $(k,l)\in\I$.
\end{df}
In fact, $\pi$ is a continuous surjective factor map from $(\tildeX,\sigma_{\tildeX})$ to $(X,\sigma)$ and $\imath$ is an injective map (not necessarily continuous) such that $\pi\circ\imath={\rm id}_X$.

Now let $X$ be a subshift with $\sigma(X)=X$.

\begin{prop}\label{lse1}
Take any point $x\in X$. Then the following properties hold.

(1) $\sigma_{\tildeX}\circ\imath(x)=\imath\circ\sigma(x)$;

(2) If $x\in Sp_l(X,\sigma)$, then $\imath(x)\in Sp_l(\tildeX,\sigma_{\tildeX})$;

(3) If $\tilde{x}\in\pi^{-1}(\{x\})\cap Sp_l(\tildeX,\sigma_{\tildeX})$, then $x\in Sp_l(X,\sigma)$.
\end{prop}

\begin{proof}
(1) By the definition of the maps $\imath$ and $\sigma_{\tildeX}$, for every $(k,l)\in\I$, we have
\[{}_k(\sigma_{\tildeX}\circ\imath(x))_l={}_k[\sigma({}_{k+1}\tilde{x}_l)]_l,\]
where ${}_{k+1}\tilde{x}_l$ is a representative of the $\stackrel{k+1,l}{\sim}$-equivalence relation class in $\imath(x)=({}_k[x]_l)_{(k,l)\in\I}$, which means we can simply take ${}_{k+1}\tilde{x}_l=x$. This immediately follows that
\[{}_k(\sigma_{\tildeX}\circ\imath(x))_l={}_k[\sigma(x)]_l={}_k(\imath\circ\sigma(x))_l.\]
Since $(k,l)\in\I$ is taken arbitrarily, $\sigma_{\tildeX}\circ\imath(x)=\imath\circ\sigma(x)$.

(2) Let $x\in Sp_l(X,\sigma)$. Then there are two distinct $y_1,y_2\in X$ with $\sigma(y_1)=\sigma(y_2)=x$. Denote $\tilde{y}_1=\imath(y_1)$ and $\tilde{y}_2=\imath(y_2)$. As $\imath$ is injective, $\tilde{y}_1\ne\tilde{y}_2$. Now by (1), one sees
\[\sigma_{\tildeX}(\tilde{y}_1)=\sigma_{\tildeX}(\imath(y_1))=\imath\sigma(y_1)=\imath(x)=\imath\sigma(y_2)=\sigma_{\tildeX}(\imath(y_2))=\sigma_{\tildeX}(\tilde{y}_2),\]
which verifies that $\imath(x)\in Sp_l(\tildeX,\sigma_{\tildeX})$.

(3) Let $\tilde{x}$ be a left special element in $\tildeX$ with $\pi(\tilde{x})=x$. Then there are distinct elements $\tilde{y}_1, \tilde{y}_2\in\tildeX$ with
\[\sigma_{\tildeX}(\tilde{y}_1)=\sigma_{\tildeX}(\tilde{y}_2)=\tilde{x}.\]
By definition of $\sigma_{\tildeX}$, one has ${}_k[\sigma({}_{k+1}(\tilde{y}_1)_l)]_l={}_k[\sigma({}_{k+1}(\tilde{y}_2)_l)]_l={}_k[{}_k\tilde{x}_l]_l$ for all $(k,l)\in\I$, which follows that
\begin{align}\label{pe}
P_l(\sigma^{k+1}({}_{k+1}(\tilde{y}_1)_l))=P_l(\sigma^{k+1}({}_{k+1}(\tilde{y}_2)_l)).
\end{align}
On the other hand, since $\tilde{y}_1\ne\tilde{y}_2$, there exists $(k_0,l_0)$ such that for all $(k,l)\in\I$ with $(k_0,l_0)\preceq(k,l)$, ${}_k[{}_k(\tilde{y}_1)_l]_l\ne {}_k[{}_k(\tilde{y}_2)_l]_l$, that is, ${}_k(\tilde{y}_1)_l\stackrel{k,l}{\nsim}{}_k(\tilde{y}_2)_l$. By \eqref{pe},
\[({}_k(\tilde{y}_1)_l)_{[0,k)}\ne ({}_k(\tilde{y}_2)_l)_{[0,k)}.\]
Denote $y_1=\pi(\tilde{y}_1)$ and $y_2=\pi(\tilde{y}_2)$. As $(y_1)_{[0,k)}$ and $(y_2)_{[0,k)}$ are determined by $({}_k(\tilde{y}_1)_l)_{[0,k)}$ and $({}_k(\tilde{y}_2)_l)_{[0,k)}$ respectively, which are distinct prefixes of length $k$, we then see $(y_1)_{[0,k)}\ne (y_2)_{[0,k)}$, and hence $y_1\ne y_2$. However, we have
\begin{align*}
x=\pi(\tilde{x})&=\pi(\sigma_{\tildeX}(\tilde{y}_1))=\sigma(\pi(\tilde{y}_1))=\sigma(y_1)\\
&=\pi(\sigma_{\tildeX}(\tilde{y}_2))=\sigma(\pi(\tilde{y}_2))=\sigma(y_2),
\end{align*}
which shows that $x$ has at least two inverse images, i.e., $x\in Sp_l(X,\sigma)$.
\end{proof}

\begin{prop}\label{lse2}
If $Sp_l(X,\sigma)\subset X$ is finite, then 
\[Sp_l(\tildeX,\sigma_{\tildeX})\cap\pi^{-1}(Sp_l(X,\sigma))\subset\imath(Sp_l(X,\sigma)).\]
\end{prop}

\begin{proof}
Let $\tilde{x}\in Sp_l(\tildeX,\sigma_{\tildeX})\cap\pi^{-1}(Sp_l(X,\sigma))$, and denote $x=\pi(\tilde{x})\in Sp_l(X,\sigma)$. Since $\tilde{x}\in Sp_l(\tildeX,\sigma_{\tildeX})$, we can choose two distinct elements $\tilde{y}_1$ and $\tilde{y}_2$ in $\tildeX$ such that
\[\sigma_{\tildeX}(\tilde{y}_1)=\sigma_{\tildeX}(\tilde{y}_2)=\tilde{x}.\]
This follows that, there exists $(k_0,l_0)\in\I$, such that for all $(k,l)$ with $(k_0,l_0)\preceq (k,l)$,
\begin{align}\label{5.14}
\begin{cases}
P_l(({}_{k+1}(\tilde{y}_1)_l)_{[k+1,\infty)})=P_l(({}_{k+1}(\tilde{y}_2)_l)_{[k+1,\infty)})=P_l(({}_k\tilde{x}_l)_{[k,\infty)}),\\
({}_{k+1}(\tilde{y}_1)_l)_{[1,k+1)}=({}_{k+1}(\tilde{y}_2)_l)_{[1,k+1)}=({}_k\tilde{x}_l)_{[0,k)}=x_{[0,k)},\\
{}_k(\tilde{y}_1)_l\stackrel{k,l}{\nsim}{}_k(\tilde{y}_2)_l.
\end{cases}
\end{align}
Note that by the third row of the above equations group, for those $(k,l)\in\I$ satisfying both $(k_0,l_0)\preceq(k+1,l)$ and $(k_0,l_0)\preceq(k,l)$, $({}_k(\tilde{y}_1)_l)_{[0,1)}\ne({}_k(\tilde{y}_2)_l)_{[0,1)}$. Denote
\[\mu=({}_k(\tilde{y}_1)_l)_{[0,1)}\in\A\ \ {\rm and}\ \ \nu=({}_k(\tilde{y}_2)_l)_{[0,1)}\in\A.\]
By the first row of \eqref{5.14}, for all these $(k,l)$, 
\[\mu\cdot\sigma({}_k(\tilde{y}_1)_l), \mu\cdot\sigma({}_k(\tilde{y}_2)_l)\in X\ {\rm and}\ \nu\cdot\sigma({}_k(\tilde{y}_1)_l), \nu\cdot\sigma({}_k(\tilde{y}_2)_l)\in X.\]
This means that $\sigma({}_k(\tilde{y}_1)_l)$ and $\sigma({}_k(\tilde{y}_2)_l)$ are (not necessarily distinct) left special elements in $X$. Since for all sufficiently large $(k,l)$, we know that ${}_k(\tilde{y}_1)_l$'s share the common $k$-prefixes with $\pi(\tilde{y}_1)$ and ${}_k(\tilde{y}_2)_l$'s share the common $k$-prefixes with $\pi(\tilde{y}_2)$, that $\sigma({}_k(\tilde{y}_1)_l)$ and $\sigma({}_k(\tilde{y}_2)_l)$ also share the common $k$-prefixes, and  that the number of the special elements in $X$ is finite, a single argument of the Pigeon Principle yields that, for all sufficiently large $(k,l)$ satisfying both $(k_0,l_0)\preceq(k+1,l)$ and $(k_0,l_0)\preceq(k,l)$, 
\[\sigma({}_k(\tilde{y}_1)_l)=z=\sigma({}_k(\tilde{y}_2)_l)\]
for some $z\in Sp_l(X,\sigma)$. This finally follows that 
\[\imath(z)=\sigma_{\tildeX}(\tilde{y}_2)=\tilde{x}=\sigma_{\tildeX}(\tilde{y}_1)=\imath(z).\]
In particular, $\tilde{x}\in\imath(Sp_l(X,\sigma))$. This proves the proposition.
\end{proof}

Now from Proposition \ref{lse1} and \ref{lse2}, we immediately get the following conclusion.

\begin{cor}\label{lse3}
For every subshift $X$, if $Sp_l(X,\sigma)$ is finite, then
\[Sp_l(\tildeX,\sigma_{\tildeX})=\imath(Sp_l(X,\sigma)).\]
Consequently, $|Sp_l(\tildeX,\sigma_{\tildeX})|=|Sp_l(X,\sigma)|$ since $\imath$ is injective.
\end{cor}

\begin{proof}
According to Proposition \ref{lse1},
\[\imath(Sp_l(X,\sigma))\subset Sp_l(\tildeX,\sigma_{\tildeX})\subset \pi^{-1}(Sp_l(X,\sigma)).\]
Further applying Proposition \ref{lse2} yields that
\begin{align*}
\imath(Sp_l(X,\sigma))&\subset Sp_l(\tildeX,\sigma_{\tildeX})\\
&=Sp_l(\tildeX,\sigma_{\tildeX})\cap\pi^{-1}(Sp_l(X,\sigma))\subset\imath(Sp_l(X,\sigma)).
\end{align*}
This finally concludes $Sp_l(\tildeX,\sigma_{\tildeX})=\imath(Sp_l(X,\sigma))$.
\end{proof}

\begin{prop}\label{lse4}
Let $X$ be a subshift such that $Sp_l(X,\sigma)$ is finite and contains no periodic point. Then for every $x\in Sp_l(X,\sigma)$, $\imath(x)$ is an isolated point in $\tildeX$.
\end{prop}

\begin{proof}
Fix $x\in Sp_l(X,\sigma)$.  By Lemma 3.2.4 in \cite{HW}, there exists $N\in\N$ such that $\sigma^N(x)$ is isolated in past equivalence. Then by Lemma 4.2 in \cite{B}, $\imath(\sigma^N(x))$ is isolated in $\tildeX$. According to (1) of Proposition \ref{lse1}, 
\[\sigma_{\tildeX}^N\circ\imath(x)=\imath\circ\sigma^N(x),\]
which follows that $\imath(x)\in\sigma_{\tildeX}^{-N}(\imath\circ\sigma^N(x))$. However, also note that $\sigma_{\tildeX}$ is a local homeomorphism, and together with the fact that $\imath(\sigma^N(x))$ is isolated in $\tildeX$, we see that $\sigma_{\tildeX}^{-N}(\imath\circ\sigma^N(x))$ is a finite clopen set in $X$, which implies that every point of it is an isolated point in $\tildeX$. In particular, $\imath(x)$ is an isolated point in $X$.
\end{proof}

Now use $I(\tildeX)$ to denote the set of isolated points in $\tildeX$. Then combining Proposition \ref{lse4}, Corollary \ref{lse3} and the aforementioned properties of $\tildeX$, we can now make a summary as follows. We note that $\sigma_{\tildeX}$ is aperiodic if so is $\sigma$, since $\pi$ is a factor map from $\tildeX$ onto $X$ intertwining $\sigma$ and $\sigma_{\tildeX}$.
\begin{thm}\label{shift}
For every aperiodic subshift $X$ with $|Sp_l(X,\sigma)|<\infty$, its cover system $(\tildeX,\sigma_{\tildeX})$ is a zero-dimensional system, where $\sigma_{\tildeX}$ is an aperiodic surjective local homeomorphism. Moreover,

(i) the cover system $(\tildeX,\sigma_{\tildeX})$ also has finitely many left special elements. In particular,
\[|Sp_l(\tildeX,\sigma_{\tildeX})|=|Sp_l(X,\sigma)|;\]

(ii) every point in $Sp_l(\tildeX,\sigma_{\tildeX})$ is isolated in $\tildeX$, that is,
\[Sp_l(\tildeX,\sigma_{\tildeX})\subset I(\tildeX).\]
\end{thm}

\subsection{Aperiodic local homeomorphisms and the nuclear dimension}

\begin{prop}\label{3.2}
Let $X$ be a compact metric space with finite covering dimension, and $T:X\to X$ an aperiodic surjective local homeomorphism. 

If $Sp_l(X,T)$ is finite, and in which every point is isolated in $X$, then $A=C^*(\mathcal{G}_{(X,T)})$ is a unital separable amenable $C^*$-algebra with finite nuclear dimension. In particular,
\begin{align}\label{nucdim}
{\rm dim}_{\rm nuc}(A)\le 6({\rm dim}(X)+1)^2.
\end{align}

Besides, ${\rm tsr}(A)\le {\rm tsr}(C^*(\mathcal{G}_{(Y,T')}))+1$, and ${\rm RR}(A)=0$ if ${\rm RR}(C^*(\mathcal{G}_{(Y,T')})=0$, where $Y=X\setminus\bigcup_{\omega\in Sp_l(X,T)}{\rm Orb}_T(\omega)$ and $T': Y\to Y$ is the restriction map $T|_Y$.
\end{prop}

\begin{proof}
Since $T$ is an aperiodic local homeomorphism on a compact metric space, its Deaconu-Renault groupoid $\mathcal{G}_{(X,T)}$ is a free, $\acute{\rm e}$tale amenable groupoid. Then its groupoid $C^*$-algebra $A$ is unital, separable, and amenable.

Set $Y=X\setminus\bigcup_{\omega\in Sp_l(X,T)}{\rm Orb}_T(\omega)$. It is clear that $Y$ is a compact subset of $X$ since every point in $Sp_l(X,T)$ is isolated in $X$. Moreover, $T|_Y(Y)=Y$ and $T|_Y$ is a homeomorphism on $Y$, since $Y$ contains no points on orbits of any left special elements. Also note that 
\[X=Y\sqcup\bigsqcup_{1\le i\le {\bf n}(X,T)}{\rm Orb}_T(\omega_i),\]
where ${\bf n}(X,T)$ is the number of distinct orbits of left special elements in $X$. As every open (isolated) subset ${\rm Orb}_T(\omega_i)$ corresponds to the full equivalence relation on a countably infinite set, whose $C^*$-algebra is a closed ideal of $A$, isomorphic to the compact operators $\mathcal{K}$, we see that, according to Proposition 4.3.2 in \cite{SW}, there is an exact sequence
\begin{align}\label{exact}
0\longrightarrow\mathcal{K}^{{\bf n}(X,T)}\longrightarrow A\longrightarrow C^*(\mathcal{G}_{(Y,T')})\to 0.
\end{align}
Since $T'$ is a homeomorphism of $Y$ onto itself, we have $C^*(\mathcal{G}_{(Y,T')})\cong C(Y)\rtimes_{T'}\Z$. It is also clear that ${\rm dim}(Y)={\rm dim}(X)$, and that $T'$ is free since $T$ is free. Then by Proposition 2.9 of \cite{WZ} and Corollary 7.6 in \cite{SWZ}, one concludes that
\[{\rm dim}_{\rm nuc}(A)\le {\rm dim}_{\rm nuc}(\mathcal{K}^{{\bf n}(X,T)})+{\rm dim}_{\rm nuc}(C(Y)\rtimes_{T'}\Z)+1\le 6({\rm dim}(X)+1)^2.\]
This verifies \eqref{nucdim}.

As for the real rank of $A$, we note that, according to Corollary 3.16 in \cite{BP}, every projection in $A/\mathcal{K}^{{\bf n}(X,T)}$ lifts to a projection in $A$. Now since ${\rm RR}(\mathcal{K}^{{\bf n}(X,T)})={\rm RR}(A/\mathcal{K}^{{\bf n}(X,T)})={\rm RR}(C^*(\mathcal{G}_{(Y,T')}))=0$ by the assumption, Theorem 3.14 in \cite{BP} then forces ${\rm RR}(A)=0$.

Finally, that ${\rm tsr}(A)\le {\rm tsr}(C^*(\mathcal{G}_{(Y,T')}))+1$ follows from Corollary 4.12 in \cite{Rie}, the fact that ${\rm tsr}(\mathcal{K}^{{\bf n}(X,T)})=1$, and the exact sequence \eqref{exact}.
\end{proof}

\begin{cor}\label{cor332}
Let $X$ be a one-sided subshift with $\sigma(X)=X$. If $\sigma$ is aperiodic and $X$ has finitely many left special elements, then
\[{\rm dim}_{\rm nuc}(\mathcal{O}_X)<\infty.\] 
In particular, this is the case for every aperiodic subshift $X$ with nonsuperlinear-growth complexity.
\end{cor}
\begin{proof}
This follows from Theorem \ref{shift},  Proposition \ref{3.2} and Lemma \ref{useful}.
\end{proof}

\section{The Rokhlin dimension of a local homeomorphism}\label{sec:4}

\subsection{The general position of a local homeomorphism}

\begin{lem}[Lemma 3.6, \cite{Lin1}]\label{lm3}
Let $U$ be an open set, and $E\subset X$ be zero-dimensional. Then for any open set $V\supset \overline{U}$, there is an open set $U'$ with $\overline{U}\subset U'\subset V$ such that $E\cap\partial U'=\varnothing$.
\end{lem}

As a corollary of Lemma \ref{lm3}, we have
\begin{lem}\label{lm3s}
Let $X$ be a zero-dimensional space, $U,V\subset X$ be open sets with $U\subset \overline{U}\subset V$. Then there is $U'$ with $\overline{U}\subset U'\subset V$ such that $\partial U'=\varnothing$, or in other words, $U'$ is clopen.
\end{lem}

\begin{thm}\label{lm4}
Let $X$ be a compact metric space with ${\rm dim}(X)=d<\infty$ and $T: X\to X$ an aperiodic local homeomorphism such that $T$ maps zero-dimensional sets to zero-dimensional sets.

Let $U\subset X$ be an open set. Then for every $k\in\N$ and every open $V$ with $\partial U\subset V$, there is an open set $U'$ with
\[U\subset U'\subset U\cup V\ \ {\rm and}\ \ \partial U'\subset  V\]
such that the family
\[\F(U',k)\stackrel{\triangle}{=}\{T^{-i}(\partial U'): i=0,1,\cdots k-1\}\]
is in general position.
\end{thm}

\begin{proof}
We prove the lemma by an induction on $k$. First, if $k=0$, we simply take $U'=U$. 

Now we assume that the lemma holds for some $k\ge0$. It suffices to prove that it also holds for $k+1$. Take arbitrarily open sets $U,V$ with $\partial U\subset V$. Denote
\[M=\sup\{d(x, X\setminus V): x\in V\}.\]
As $X$ is a compact space and $V$ is open, $M$ is a well-defined positive number. For any integer $r$, define
\[V_r=\left\{x\in X: d(x,X\setminus V)\ge \frac{M}{r}\right\}.\]
Then $V_r$ is a compact subset of $X$. It is also clear that $V=\bigcup_{r\ge0}V_r$ since $V$ is open. By the induction hypothesis, we can choose an open set $W_0\subset X$ with
\[U\subset W_0\subset U\cup V\ \ {\rm and}\ \ \partial W_0\subset V\]
such that 
\[\F(W_0, k)=\{T^{-i}(\partial W_0): i=0,1,\cdots,k-1\}\]
is in general position.

Since $V$ is open, for every $x\in \partial W_0\subset V$, we can choose a positive number $\delta_x>0$ such that $\overline{B(x,\delta_x)}\subset V$ and the sets
\[B(x,\delta_x), T^{-1}(B(x,\delta_x)),\cdots,T^{-k}(B(x,\delta_x))\]
are pairwise disjoint. Note that this could be done due to the aperiodicity of $T$. Denote
\[B_x=B(x,\delta_x/2)\ \ {\rm and}\ \ \hat{B}_x=B(x,\delta_x).\]
Set
\[C_r=\left\{x\in\partial W_0: \frac{M}{r+1}\le d(x, X\setminus V)\le \frac{M}{r}\right\}.\]
Every $C_r$ is obviously a compact subset of $\partial W_0$ and $\{B_x\}_{x\in C_r}$ forms an open cover of $C_r$. Choose a finite number of $B_x$ with $x\in C_r$ that cover $C_r$. Also note that $\partial W_0=\bigcup_{r\ge 1} C_r$. This gives us a countable subcover of $\{B_x\}_{x\in \partial W_0}$, which we denote by $B_i\,(i\ge1)$, that covers $\partial W_0$. Besides, for every $x\in C_r$, the open sets $B_x$ used to form a finite subcover of $C_r$ all have diameter no more that $M/r$, since the distance of every point in $C_r$ to $X\setminus V$ is no more than $M/r$. This follows that
\[\lim_{i\to\infty}{\rm diam}(B_i)=0.\]
Furthermore, for every $l\ge0$, we can find a sufficiently large integer $N_l\in\N$ such that for every $x\in C_{N_l}$, $d(x, V_l)\ge 4M/N_l$. Since the diameter of every $\hat{B}_i$ of the form $\hat{B}_x$ for which $x\in C_{N_l}$ has diameter no more than $2M/N_l$, this means that there exists $r_l\in\N$ for which
\[V_l\cap\bigcap_{i\ge r_l}\hat{B}_i=\varnothing.\]
Now we will recursively construct a sequence of open sets $\{W_i\}_{i\ge0}$ starting from $W_0$ defined before, satisfying the following properties.

(P1) $W_i\subset W_0\cup \bigcup_{j\ge0}B_j$;

(P2) $W_{i-1}\subset W_i$;

(P3) $\partial W_i\subset V$;

(P4) $W_i\setminus \hat{B}_i=W_{i-1}\setminus \hat{B}_i$;

(P5) The family 
\[\F(W_i, k)\cup\left\{T^{-k}\left(\partial W_i\cap \bigcup_{1\le j\le i} B_j\right)\right\}\]
is in general position.

We now show that, the lemma follows once we are given this sequence. For this, define
\[U'=\bigcup_{i\ge 0}W_i.\] 
By (P2), we see that $U\subset W_0\subset\bigcup_{i\ge0}W_i=U'$. By (P1), we have
\[U'\subset W_0\cup\bigcup_{j\ge0}B_j\subset W_0\cup V\subset U\cup V\cup V=U\cup V.\]
This verifies $U\subset U'\subset U\cup V$. We now claim that $\partial U'\subset V$. Take arbitrarily $x\in \partial U'$. Then either it is in $\partial W_i$ for some $i$, or is the limit of a sequence $(x_l)_{l\ge1}$ with $x_l\in W_{n_l}\setminus W_{n_l-1}$ with $n_l$ strictly increasing. If the latter holds, since $W_i\setminus W_{i-1}\subset \hat{B}_i$ by (P4), we then have
\[x\in \bigcap_{r\ge0}\overline{\bigcup_{l\ge r}\hat{B}_{n_l}}.\]
In fact, this will follows that $x\in \partial W_0$. If not, since $\partial W_0$ is compact, we can find some small $\eta>0$, such that
\[\inf\{d(x, y): y\in\partial W_0\}>\eta.\]
We can then choose a sufficiently large $R$ such that $\hat{B}_{n_l}$'s have radii no more than $\eta/100$ for all $l\ge R$. Since $x\in\overline{\bigcup_{l\ge R}\hat{B}_{n_l}}$, $x$ is a limit point of $\bigcup_{l\ge R}\hat{B}_{n_l}$, therefore one can take $y\in \hat{B}_{n_L}$ with $d(x,y)<\eta/100$ for some $L\ge R$. Note that the center of $\hat{B}_{n_L}$ is in $\partial W_0$ and the radius of $\hat{B}_{n_L}$ is no more than $\eta/100$ as assumed, this will follows
\[d(x,\partial W_0)\le \eta/50,\]
a contradiction. Consequently, $x\in \partial W_0$, which implies
\[\partial U'\subset \bigcup_{i\ge 0}\partial W_i\cup\partial W_0\subset V,\]
where the second inclusion follows from (P3).

Now we shall verify that 
\[\F(U', k+1)=\{\partial U', T^{-1}(\partial U'), \cdots, T^{-k}(\partial U')\}\]
is in general position. Now according to (P4), we have, for every $r<r'$,
\begin{align*}
 \textstyle W_{r'}\setminus\bigcup_{i\ge r+1}\hat{B}_i\subset W_{r'}\setminus \hat{B}_{r'}=W_{r'-1}\setminus \hat{B}_{r'}\subset W_{r'-1},
\end{align*}
and therefore $\textstyle W_{r'}\setminus \bigcup_{i\ge r+1}\hat{B}_i\subset W_{r'-1}\setminus \bigcup_{i\ge r+1}\hat{B}_i$. On the other hand, we know from (P2) that, $W_{r'-1}\subset W_{r'}$, which deduces that
\[\textstyle W_{r'}\setminus \bigcup_{i\ge r+1}\hat{B}_i=W_{r'-1}\setminus \bigcup_{i\ge r+1}\hat{B}_i.\]
By applying this procedure recursively on $r'-1, r'-2,\cdots, r+1$, we get
\[\textstyle W_{r'}\setminus \bigcup_{i\ge r+1}\hat{B}_i=W_r\setminus \bigcup_{i\ge r+1}\hat{B}_i,\]
and so $U'\setminus \bigcup_{i\ge r+1}\hat{B}_i=W_r\setminus \bigcup_{i\ge r+1}\hat{B}_i$ for every $r$. By the definition of $B_i$'s, for every $l\ge1$, we can choose a sufficiently large $r$, such that for every $i\ge r$, $\hat{B}_{i}\cap V_{l+1}=\varnothing$, which implies that 
\[U'\cap V_{l+1}=W_r\cap V_{l+1}.\]
Note that $V_l\subset (V_{l+1})^\circ$, one concludes that $V_l\cap\partial U'=V_l\cap\partial W_r$. We can take $r$ large enough with $V_l\subset \bigcup_{1\le i\le r}B_i$. By (P5), we know that
\[\textstyle\left\{\partial W_r, T^{-1}(\partial W_r), \cdots, T^{-(k-1)}(\partial W_r), T^{-k}\left(\partial W_r\cap \bigcup_{1\le j\le r}B_j\right)\right\}\]
is in general position. Also note that the same is true if we replace each set of the above family with a smaller one, therefore
\begin{align*}
\textstyle&\left\{V_l\cap\partial W_r, T^{-1}(V_l\cap\partial W_r), \cdots, T^{-(k-1)}(V_l\cap\partial W_r), T^{-k}\left(V_l\cap\partial W_r\right)\right\}\\
=&\left\{V_l\cap\partial U', T^{-1}(V_l\cap\partial U'), \cdots, T^{-(k-1)}(V_l\cap\partial U'), T^{-k}\left(V_l\cap\partial U'\right)\right\}
\end{align*}
is in general position. Since this holds for every $l$, and since $T$ is continuous and $V_l\cap\partial U'$ is closed, $T^{-i}(V_l\cap\partial U')$ is also closed. Then  according to the fact that $X$ is zero-dimensional and $V=\bigcup_{l\ge0}V_l$, one sees that the family
\[\textstyle\left\{V\cap\partial U', T^{-1}(V\cap\partial U'), \cdots, T^{-(k-1)}(V\cap\partial U'), T^{-k}\left(V\cap\partial U'\right)\right\}\]
is in general position. On the other hand, we have mentioned that $\partial U'\subset V$, hence is the above family is the same as $\F(U',k+1)$. This finishes the verification of the claim that $\F(U',k+1)$ is indeed in general position.

Now it suffices to show how to construct $W_{i+1}$ from $W_i$, so that the proof will then be completed. For every finite subfamily $\textstyle S\subset \F(W_i, k)\cup\{T^{-k}(\partial W_i\cap \bigcup_{1\le j\le i} B_j)\}\stackrel{\triangle}{=}\tilde{\F}_i$ for which $\cap S\ne\varnothing$, according to (D4) in Sect.\ref{sec:2}, there is a zero-dimensional set $Y_S\subset \cap S$, which is $F_\sigma$ in $\cap S$ such that
\[{\rm dim}((\cap S)\setminus Y_S)={\rm dim}(\cap S)-1.\]
Note that, if ${\rm dim}(X)=0$, we can just simply set $Y_S=\cap S$.
Also note that since $\cap S$ is $\sigma$-compact, and every $F_\sigma$ subset of a $\sigma$-compact space is still $\sigma$-compact, we see that $Y_S$ is $\sigma$-compact. Define
\[E=\bigcup_{S\subset \tilde{\F}_i, 0\le j\le k}T^{j}(Y_S)\]
where $S$ is taken over all finite subfamilies of $\tilde{\F}_i$ for which $\cap S\ne\varnothing$. Now we can see that, for every $S\subset \tilde{\F}_i$ and $j\in\{0,1,\cdots,k\}$, $T^{j}(Y_S)$ is $\sigma$-compact and zero-dimensional, as $T$ is continuous and maps zero-dimensional sets to zero-dimensional sets. Since $\tilde{F}_i$ is a finite family, $E$ is a finite union of zero-dimensional $\sigma$-compact sets, which follows that $E$ itself is both $\sigma$-compact and zero-dimensional. By Lemma \ref{lm3}, there is an open set $O$ with $\overline{B_{i+1}\cap W_i}\subset O$ such that
\[\overline{O}\subset \hat{B}_{i+1}\cap\left(U\cup\bigcup_{j\ge0}B_j\right)\cap V\ \ {\rm and}\ \ E\cap\partial O=\varnothing.\]
Now we define
\[W_{i+1}=W_i\cup O.\]
It is then immediate that $W_{i+1}$ satisfies (P1)--(P4), and it only remains to be shown that $W_{i+1}$ satisfies (P5) as well.

For this, let us assume by contradiction that (P5) is not satisfied. Then there exists a subfamily $S\subset \tilde{F}_{i+1}$ of cardinality $m$ such that 
\[{\rm dim}(\cap S)>\max\{-1,d-m\}.\]
Write $S=\{S_1,S_2,\cdots, S_m\}$. Then each $S_l$ is an element of the form $T^{-j_l}(\partial W_{i+1})$ for some $0\le j_l\le k-1$ or else $T^{-k}(\partial W_{i+1}\cap\bigcup_{1\le j\le i+1}B_j)$. Note that $j_l\ne j_{l'}$ for $l\ne l'$, since $S_l$'s are distinct elements. Now from $W_{i+1}=W_i\cup O$, we have
\[\partial W_{i+1}\subset (\partial W_i\setminus O)\cup\partial O.\]
Therefore, for any $l$ with $j_l\le k-1$, one has
\[S_l\subset T^{-j_l}(\partial W_i)\cup T^{-j_l}(\partial O).\]
For the case of $j_l=k$, we also have the following similar inclusion
\begin{align*}
S_l&=T^{-k}\left(\partial W_{i+1}\cap\bigcup_{1\le j\le i+1}B_j\right)\\
&\subset T^{-k}\left((\partial W_i\setminus O)\cap\bigcup_{1\le j\le i+1}B_j\right)\cup T^{-k}(\partial O)\\
&\xlongequal[]{\overline{B_{i+1}\cap W_i}\subset O} T^{-k}\left((\partial W_i\setminus O)\cap\bigcup_{1\le j\le i}B_j\right)\cup T^{-k}(\partial O)\\
&\subset T^{-k}\left((\partial W_i)\cap\bigcup_{1\le j\le i}B_j\right)\cup T^{-k}(\partial O).
\end{align*}
Combining these two cases, we see that $S_l$ is a subset of $S_l^0\cup S_l^1$, where $S_l^1=T^{-j_l}(\partial O)$ and 
\begin{equation}
S_l^0=
\begin{cases}
T^{-j_l}(\partial W_i),\ \ &0\le j_l\le k-1,\\
T^{-k}((\partial W_i)\cap\bigcup_{1\le j\le i}B_j),\ \ &j_l=k.
\end{cases}
\end{equation}
Thus $\cap S$ is a subset of the following finite union
\[\cap S\subset \bigcup_{\af\in\{0,1\}^m}\left(\bigcap_{1\le l\le m}S_l^{\af_l}\right).\]
Since ${\rm dim}(\cap S)>\max\{-1,d-m\}$, at least one of these intersections has dimension larger than $\max\{-1,d-m\}$. However, since $\overline{O}\subset \hat{B}_{i+1}$ and $\hat{B}_{i+1}, T^{-1}(\hat{B}_{i+1}), \cdots, T^{-k}(\hat{B}_{i+1})$ are pairwise disjoint by choice of $\delta_x$'s, any index $\af$ with at least two $\af_l$'s being $1$ satisfies that $\bigcap_{1\le l\le m}S_l^{\af_l}=\varnothing$. On the other hand, for $\af=(0,0,\cdots,0)$, each $S_l^{\af_l}$ is a distinct element of the family $\tilde{\F}_i$, and since this family is in general position by assumption, we see that
\[{\rm dim}\left(\bigcap_{1\le l\le m}S_l^0\right)>\max\{-1,d-m\}.\]
Note that it has nothing to do with the order of the $S_l$'s, that is, we may assume without loss of generality that
\[{\rm dim}\left(S_1^1\cap\bigcap_{2\le l\le m}S_l^0\right)>\max\{-1,d-m\},\]
but this will also lead to a contradiction. To see this, denote $\hat{S}=\{S_2^0,\cdots, S_m^0\}$. From $\hat{S}\subset\tilde{\F}_i$, we know that
\[{\rm dim}(\cap \hat{S}\setminus Y_{\hat{S}})\le {\rm dim}(\cap\hat{S})-1\le \min\{-1,d-m\}.\]
On the other hand, since $E\cap\partial O=\varnothing$ and $T^{j_1}(Y_{\hat{S}})\subset E$, we have $T^{j_1}(Y_{\hat{S}})\cap\partial O=\varnothing$, or in other words,  $T^{-j_1}(\partial O)\cap Y_{\hat{S}}=\varnothing$ and hence $T^{-j_1}(\partial O)\subset X\setminus Y_{\hat{S}}$. Therefore, 
\[S_1^1\cap \bigcap_{2\le l\le m}S_l^0\subset \cap\hat{S}\setminus Y_{\hat{S}},\]
which immediately follows that
\[\max\{-1,d-m\}<{\rm dim}\left(S_1^1\cap \bigcap_{2\le l\le m}S_l^0\right)\le {\rm dim}(\cap\hat{S}\setminus Y_{\hat{S}})\le \max\{-1,d-m\}.\]
a contradiction. The proof is now completed.
\end{proof}

\subsection{Technical lemmas}

\begin{lem}\label{lm5}
Let $X$ be a compact metric space, $T: X\to X$ be a continuous map and $E\subset X$ a closed subset. If $E\cap T^{-1}(E)=\varnothing$, then there is $\rho>0$ such that
\[\overline{B_\rho(E)}\cap T^{-1}(\overline{B_\rho(E)})=\varnothing.\]
\end{lem}

\begin{proof}
Suppose by contradiction that for any $n>0$, there is $x_n, y_n\in\overline{B_{1/n}(E)}$ with 
\[x_n=T(y_n).\]
Since $X$ is compact, by passing to a subsequence, we may assume without loss of generality that $x_n\to x$ and $y_n\to y$ as $n\to\infty$ for some $x,y\in X$. As $E$ is compact, for any $n\ge1$, we can take $x'_n,y'_n\in E$ with 
\[d(x_n,x'_n)\le\frac{1}{n}\ \ {\rm and}\ \ d(y_n,y'_n)\le\frac{1}{n}.\]
This follows that $x'_n\to x$ and $y'_n\to y$ as $n\to\infty$. Since $E$ is closed, we have $x,y\in E$. Also note that, by the continuity of $T$, $x=T(y)$. This contradicts the assumption that $E\cap T^{-1}(E)=\varnothing$.
\end{proof}

Although the following Lemma is quite simple, we still decide to mention it, since one should be careful when dealing with continuous maps which are not homeomorphisms.

\begin{lem}\label{lmc}
Let $X$ be a compact metric space and $T:X\to X$ be continuous. Then for every $A\subset X$, $T(\overline{A})=\overline{T(A)}$.

For any $m,n\in\N$ and $A\subset X$, one has
\[T^{-m-n}(A)=T^{-n}(T^{-m}(A)),\]
and if $T$ is surjective, then
\[T^m(T^{-n}(A))=T^{m-n}(A).\]
\end{lem}
\begin{proof}
Since $X$ is compact and $\overline{A}$ is closed, $\overline{A}$ is compact, which follows that $T(\overline{A})$ is compact, and hence closed. Then from $T(A)\subset T(\overline{A})$, we have $\overline{T(A)}\subset T(\overline{A})$. Conversely, that $T(\overline{A})\subset \overline{T(A)}$ follows directly from the continuity of $T$.

Let $x\in T^{-m-n}(A)$. This is equivalent to $T^m(T^n(x))=T^{m+n}(x)\in A$, which is also equivalent to $T^n(x)\in T^{-m}(A)$, and consequently $x\in T^{-n}(T^{-m}(A))$. 

Now we consider the last equality. 

$\bullet$ If $m\ge n$: Let $x\in T^{m-n}(A)$. Then we can find $y\in A$ with $x=T^{m-n}(y)$. Take $z\in T^{-n}(\{y\})\subset T^{-n}(A)$ since $T$ is surjective. Then $x=T^{m-n}(y)=T^{m-n}(T^n(z))=T^m(z)$, that is, $x\in T^m(T^{-n}(A))$. Conversely, let $x\in T^m(T^{-n}(A))$. Then there is $z\in T^{-n}(A)$ with $x=T^m(z)$. Write $y=T^n(z)\in A$. We then have $x=T^{m-n}(y)\in T^{m-n}(A)$;

$\bullet$ If $m<n$: Let $x\in T^{m-n}(A)$. This is equivalent to $T^{n-m}(x)\in A$. Find an element $z\in T^{-m}(\{x\})\subset T^{-m}(T^{m-n}(A))=T^{-n}(A)$. Then one has $x=T^m(z)\in T^m(T^{-n}(A))$. Conversely, let $x\in T^m(T^{-n}(A))$. Take $z\in T^{-n}(A)$ with $x=T^m(z)$. We then have that $T^{n-m}(x)=T^{n-m}(T^m(z))=T^n(z)\in A$. The proof is finished.
\end{proof}



\begin{lem}\label{lm7}
Let $X$ be a zero-dimensional compact metric space and $T: X\to X$ an aperiodic local homeomorphism.
Fix $N\in\N$ and $m=2N$. Let $U$ be a clopen set and $V$ be an open set in $X$. If $U,V$ satisfies the following properties:

(1) The sets
\[T^{-2N+1}(\overline{U}), T^{-2N+2}(\overline{U}),\cdots,T^{-N}(\overline{U})\]
are pairwise disjoint and $T^{-i}(T^i(U))=U$ for $i=1,2,\cdots,N-1$;

(2) The sets
\[T^{-2N+1}(T^N(\overline{V})),\cdots,T^{-N}(T^N(\overline{V}))\]
are pairwise disjoint;

(3) There is $\eta>0$, such that for every $x\in\overline{V}$, $y\in B(x,\eta)$ and every $1\le i\le N-1$, 
\[T^{-i}\circ T^{i+N}(y)=\{T^N(y)\}.\]

Then there is a clopen set $W\subset X$ with
\begin{align*}
U\subset W,\ \  V\subset\bigcup_{0\le i\le m-1}T^{-i}(W),\ \ T^{-i}(T^i(W))=W
\end{align*}
for $i=1,2,\cdots, N-1$, and the sets 
\[T^{-2N+1}(\overline{W}), T^{-2N+2}(\overline{W}),\cdots, T^{-N}(\overline{W})\]
are pairwise disjoint.
\end{lem}

\begin{proof}
Define
\[\textstyle R=\overline{V}\setminus\bigcup_{0\le i\le m-1}T^{-i}(U).\]
Then $R$ is clearly a closed set.  Also note that since $R\subset\overline{V}$, the following sets
\[T^{-2N+1}(T^N(R)),\cdots,T^{-N}(T^N(R))\]
are then mutually disjoint. Let $\eta>0$ be the positive number satisfying $(3)$. By Lemma \ref{lm5}, we may take $\rho_1\in(0,\eta)$ such that the sets
\[T^{-2N+1}(\overline{B_{\rho_1}(T^N(R))}),\cdots,T^{-N}(\overline{B_{\rho_1}(T^N(R))})\]
are pairwise disjoint. Since $T$ is continuous, take further $\rho\in(0,\rho_1)\subset(0,\eta)$ such that
\[T^N(\overline{B_\rho(R)})\subset \overline{B_{\rho_1}(T^N(R))}.\]
Note that the following sets
\[T^{-2N+1}(T^N(\overline{B_\rho(R)})),\cdots,T^{-N}(T^N(\overline{B_\rho(R)}))\]
are then mutually disjoint.

{\bf Claim 1.} We now claim that, there exists a positive number  $\delta_1\in(0,\rho)$ such that for every $x\in R$ and $t\in\{N+1,\cdots,2N-1\}$,
\begin{align}\label{eq0}
T^t(\overline{B(x,\delta_1)})\cap\overline{U}=\varnothing.
\end{align}
In fact, if such $\delta$ doesn't exist, then for every $n\ge1$, we can find $x_n\in R$, $y_n\in X$ and $t_n\in\{N+1,\cdots,2N-1\}$ such that
\[y_n\in T^{t_n}(\overline{B(x_n, 1/n)})\cap\overline{U}.\]
By a compactness argument, we may assume $x_n\to x$ as $n\to\infty$ for some $x\in X$. Since $t_n$ takes value in a finite set, we may also assume $t_n=t$ by the Pigeon Principle. Since $R$ is closed, $x\in R$. Find $y'_n\in\overline{B(x_n,1/n)}$ with $y_n=T^t(y'_n)$. Note that $\lim_{n\to\infty}y'_n=x$. This then follows that
\[\lim_{n\to\infty}y_n=T^t(x)\in\overline{U}.\]
Then we have $x\in T^{-t}(\overline{U})\cap R$. However, since $U$ is clopen, one sees that
\begin{align*}
R\cap T^{-t}(\overline{U})\subset T^{-t}(\overline{U})\setminus T^{-t}(U)=T^{-t}(\partial U)=\varnothing
\end{align*}
which is a contradiction. The claim then follows.

{\bf Claim 2.} Similarly, we also claim that, there exists a positive number $\delta_2\in(0,\rho)$ such that for every $x\in R$ and $t\in\{1,\cdots,N-1\}$,
\begin{align}\label{eq0s}
T^N(\overline{B(x,\delta_2)})\cap T^t(\overline{U})=\varnothing.
\end{align}
Assume not, that there are sequences $x_n\in R$ and $y_n'\in \overline{B(x_n,1/n)}$ with $x_n\to x$ as $n\to\infty$ for some $x\in R$ and
\[y_n=T^N(y'_n)\in T^t(\overline{U}).\]
Note that $\lim_{n\to\infty}y'_n=x$, which follows that $\lim_{n\to\infty}y_n=T^N(x)\in T^t(\overline{U})$. Therefore,
\[x\in T^{-N}(T^t(\overline{U}))=\overline{T^{-N}(T^t(U))}=\overline{T^{-N+t}(T^{-t}(T^t(U)))}.\]
Note that the last equality holds because we always have $T^{-n}\circ T^{-m}=T^{-n-m}$ for $n,m<0$. According to the assumption that $T^{-t}(T^t(U))=U$ for all $t\in\{1,\cdots,N-1\}$, we then have $x\in\overline{T^{-N+t}(U)}=T^{-N+t}(\overline{U})$, and consequently, since $U$ is clopen,
\begin{align*}
x\in R\cap T^{-N+t}(\overline{U})\subset T^{-N+t}(\overline{U})\setminus T^{-N+t}(U)=\varnothing.
\end{align*}
This is again a contradiction. The claim is then proved.

Set $\delta=\min\{\delta_1,\delta_2\}\in(0,\rho)$. Then $\{B(z,\delta)\}_{z\in R}$ forms an open cover of $R$. Note that according to Lemma \ref{lm3s}, we may also choose a clopen set $U_z$ with
\[\overline{B(z,\delta/2)}\subset U_z\subset B(z,\delta),\]
and all such $U_z$'s clearly satisfy the above two claims and form an open cover of $R$ as well. Therefore, without loss of generality, we may assume each $B(z,\delta)$ is clopen. Now since $R$ is compact, we may fix a finite subcover
\[\mathcal{U}=\{B(z_i,\delta)\}_{1\le i\le s}.\]
Denote $\overline{\mathcal{U}}=\{\overline{O}: O\in\mathcal{U}\}$, and define
\[C=\bigcup_{E\in\overline{\mathcal{U}}}E.\]
It is then immediate that $R\subset C\subset \overline{B_{\rho}(R)}$ since $0<\delta<\rho$. This follows that the sets
\begin{align}\label{eqc}
T^{-2N+1}(T^N(C)),\cdots,T^{-N}(T^N(C))
\end{align}
are also pairwise disjoint. Now we define
\[W=U\cup\bigcup_{H\in\mathcal{U}}T^N(H).\]
Note that by Lemma \ref{lmc} and since $\mathcal{U}$ is finite, we then have
\[\overline{W}=\overline{U}\cup\bigcup_{E\in\overline{\mathcal{U}}}T^N(E).\]
We shall then check that $W$ satisfies all of the desired properties.

(1) $W$ is clopen: This is because $T^N$ is a local homeomorphism, hence is an open map, $U$ and $H$ are clopen sets, and $\mathcal{U}$ is finite;

(2) $U\subset W$: This is immediate from the construction of $W$;

(3) $V\subset\bigcup_{1\le i\le m}T^{-i}(W)$: Since $T^N(H)\subset W$, we have $H\subset T^{-N}(W)$, and hence 
\[\bigcup_{H\in\mathcal{U}}H\subset \bigcup_{0\le i\le m-1}T^{-i}(W).\]
Consequently, we have
\begin{align*}
V\subset \overline{V}&\subset R\cup\bigcup_{0\le i\le m-1}T^{-i}(U)\\
&\subset\bigcup_{H\in\mathcal{U}}H\cup\bigcup_{0\le i\le m-1}T^{-i}(U)\\
&\subset\bigcup_{H\in\mathcal{U}}H\cup\bigcup_{0\le i\le m-1}T^{-i}(W)=\bigcup_{0\le i\le m-1}T^{-i}(W).
\end{align*}

(4) $T^{-i}(T^i(W))=W$: First, for every $i=1,2\cdots,N-1$, one sees that
\[T^{-i}(T^i(W))=T^{-i}(T^i(U))\cup\bigcup_{H\in\mathcal{U}}T^{-i}(T^{i+N}(H)).\]
By assumption, $T^{-i}(T^i(U))=U$. Now it suffices to show that for every $H\in\mathcal{U}$,
\[T^{-i}(T^{i+N}(H))=T^N(H).\]
It is obvious that $T^N(H)\subset T^{-i}(T^{i+N}(H))$. Conversely, let $y\in H$ arbitrarily. Write $H=B(z,\delta)$ for some $z\in R$. Note that $z\in \overline{V}$ and $\delta<\rho<\eta$, which by hypothesis follows that
\[T^{-i}(T^{i+N}(y))=\{T^N(y)\}.\]
This concludes $T^{-i}(T^{i+N}(H))\subset T^N(H)$, which then establishes the equality.

(5) Now it only remains to be shown that
\[T^{-2N+1}(\overline{W})\cdots, T^{-N}(\overline{W})\]
are pairwise disjoint. For this, we assume by contradiction that this is not the case. Then there are $i,j\in\{1,\cdots,N\}$ with $i<j$ such that there exists $x\in X$ with
\[x\in T^{-2N+i}(\overline{W})\cap T^{-2N+j}(\overline{W})\ne\varnothing.\]
This is equivalent to
\[y_1=T^{2N-i}(x)\in\overline{W}\ \ {\rm and}\ \ y_2=T^{2N-j}(x)\in\overline{W}.\]
Note that we then have $y_1=T^{j-i}(y_2)$. We divide $y_1,y_2\in\overline{W}$ into following cases:

$\bullet$ $y_1\in\overline{U}$ and $y_2\in\overline{U}$: This follows that 
\[x\in T^{-(2N-i)}(\overline{U})\cap T^{-(2N-j)}(\overline{U})\ne\varnothing.\]
Since $-2N+1\le-(2N-i),-(2N-j)\le -N$, this contradicts the assumption on $U$;

$\bullet$ $y_1,y_2\in T^N(E)$ for some $E\in\overline{\mathcal{U}}$:  Note that $E\subset C$, which follows that 
\[x\in T^{-(2N-i)}(T^N(C))\cap T^{-(2N-j)}(T^N(C))\ne\varnothing.\]
This contradicts \eqref{eqc};

$\bullet$ $y_1\in T^N(E_1)$ and $y_2\in T^N(E_2)$ for some $E_1\ne E_2$ in $\overline{\mathcal{U}}$: This follows that
\[x\in T^{-(2N-i)}(T^N(C))\cap T^{-(2N-j)}(T^N(C))\ne\varnothing,\]
which again contradicts \eqref{eqc};

$\bullet$ $y_1\in\overline{U}$ and $y_2\in T^N(E)$ for some $E\in\overline{\mathcal{U}}$: As aforementioned, $y_1=T^{j-i}(y_2)$, which follows that
\[y_1\in T^{j-i}(T^N(E))\cap\overline{U}=T^{N+j-i}(E)\cap\overline{U}\ne\varnothing.\]
Note that $N+j-i\in\{N+1,\cdots,2N-1\}$, therefore this contradicts \eqref{eq0};

$\bullet$ $y_1\in T^N(E)$ and $y_2\in \overline{U}$ for some $E\in\overline{\mathcal{U}}$: This follows that
\[y_1\in T^{j-i}(\overline{U})\cap T^N(E)\ne\varnothing.\]
Note that $j-i\in\{1,\cdots,N-1\}$. Then this contradicts \eqref{eq0s}. 

The proof is now finished.
\end{proof}




\subsection{The Rokhlin dimension of an aperiodic local homeomorphism}

\begin{thm}\label{zdr}
Let $X$ be a zero-dimensional compact metric space and $T: X\to X$ an aperiodic surjective local homeomorphism. Suppose that
\[|Sp_l(X,T)|<\infty.\]
Then $(X,T)$ has finite topological Rokhlin dimension. In particular,
\[{\rm dim}_{\rm Rok}(X,T)\le 2|Sp_l(X,T)|+1.\]
\end{thm}
\begin{proof}
First we note that, for every $x\in X$ and $i,j\in\N$ with $i<j$,
\[T^{-i}(T^N(x))\cap T^{-j}(T^N(x))=\varnothing.\]
For this, suppose that $y\in T^{-i}(T^N(x))\cap T^{-j}(T^N(x))$, then $T^i(y)=T^N(x)=T^j(y)$. This follows that
\[T^{N+j}(x)=T^j(T^N(x))=T^{i+j}(y)=T^i(T^N(x))=T^{N+i}(x)\]
and hence
\[T^{j-i}(T^{N+i}(x))=T^{N+i}(x),\]
which means $T^{N+i}(x)$ is a periodic point, a contradiction. 

Now by Lemma \ref{lm5}, there exists $\rho_x>0$ such that
\[T^{-3N+1}(\overline{B(T^N(x),\rho_x)}),\cdots, T^{-N}(\overline{B(T^N(x),\rho_x)})\]
are mutually disjoint. We then find $\rho'_x>0$ such that $T^N(\overline{B(x,\rho'_x)})\subset\overline{B(T^N(x),\rho_x)}$. This will follow that the sets
\begin{align}\label{dis}
T^{-3N+1}(T^N(\overline{B(x,\rho'_x)})),\cdots,T^{-N}(T^N(\overline{B(x,\rho'_x)}))
\end{align}
are pairwise disjoint. Write $Sp_l(X,T)=\{\omega_1,\cdots,\omega_q\}$ for $q=|Sp_l(X,T)|$ and define
\[C=\bigcup_{0\le i\le 2N-1}T^{-i}(Sp_l(X,T))=\bigcup_{1\le k\le q}\bigcup_{0\le i\le 2N-1}T^{-i}(\{\omega_k\}).\]
Since every $T^i$ are a composition of local homeomorphisms, it itself is a local homeomorphism, which follows that $T^{-i}(\{\omega_k\})$'s are discrete, and therefore finite. Consequently, $C$ is a finite set. Then we take $\rho''_x>0$ such that
\begin{align}\label{eqt}
B(x,\rho''_x)\subset X\setminus C.
\end{align}
Define $\delta_x=\min\{\rho_x,\rho'_x,\rho''_x\}>0$. Applying Lemma \ref{lm3s}, there is a clopen set $U_x$ with
\begin{align}\label{eqd}
x\in\overline{B(x,\delta_x/2)}\subset U_x\subset B(x,2\delta_x/3).
\end{align}
We now claim that every $U_x$ satisfies the following property: 

$\bullet$ $T^{-2N+1}(\overline{U_x}),\cdots,T^{-N}(\overline{U_x})$ are mutually disjoint: This is because we have inclusions
\begin{align*}
T^{-2N+i}(\overline{U})&\subset T^{-2N+i}(\overline{B(x,\delta_x}))\\
&\subset T^{-2N+i}(\overline{B(x,\rho'_x}))\subset T^{-3N+i}(T^N(\overline{B(x,\rho'_x})))
\end{align*}
for all $i=1,2,\cdots,N$. Then by \eqref{dis}, these sets are mutually disjoint;

$\bullet$ $T^{-i}(T^i(U_x))=U_x$ for all $i=1,2,\cdots,N-1$: It is clear that $U_x\subset T^{-i}(T^i(U_x))$. Now let $z\in T^{-i}(T^i(U_x))$, or in other words, $T^i(z)\in T^i(U_x)$. 

We first claim that $T^i(U_x)\cap Sp_l(X,T)=\varnothing$. In fact, we have
\begin{align*}
U_x\subset B(x,\delta_x)&\subset B(x,\rho''_x)\\
&\subset X\setminus C\subset X\setminus T^{-i}(Sp_l(X,T))
\end{align*}
according to \eqref{eqt}. This then means that $T^i(U_x)\cap Sp_l(X,T)=\varnothing$. The claim follows.

Note that $T^i(U_x)\cap Sp_l(X,T)=\varnothing$ is equivalent to saying that $T^i(U_x)$ contains no left special elements. Then 
\[T^{-1}(T^i(z))=T^{i-1}(z)\ \ {\rm and}\ \ T^{-1}(T^i(U_x))=T^{i-1}(U_x).\]
This then implies $T^{i-1}(z)\in T^{i-1}(U_x)$. By applying this argument repeatedly, we finally get $z\in U_x$. Now the equality $T^{-i}(T^i(U_x))=U_x$ is established;

$\bullet$ $T^{-2N+1}(T^N(\overline{U_x})),\cdots,T^{-N}(T^N(\overline{U_x}))$ are pairwise disjoint: This is similar to the first property, since $\overline{U_x}\subset \overline{B(x,\rho'_x)}$ and the sets
\[T^{-2N+1}(T^N(\overline{B(x,\rho'_x)})),\cdots,T^{-N}(T^N(\overline{B(x,\rho'_x)}))\]
are pairwise disjoint by \eqref{dis};

$\bullet$ There is $\eta>0$ such that for every $y\in\overline{U_x}$, $z\in B(y,\eta)$ and every $1\le i\le N-1$, $T^{-i}(T^{i+N}(z))=T^N(z)$: Let 
\[\eta=\delta_x/6.\]
Let $y\in\overline{U_x}$, $z\in B(y,\eta)$ and $1\le i\le N-1$. Then $d(y,z)<\delta_x/6$. Note that since $y\in\overline{U_x}\subset \overline{B(x,2\delta_x/3)}$,
\[d(z,x)\le d(z,y)+d(y,x)\le\delta_x/6+2\delta_x/3<\delta_x,\]
which follows that $z\in B(x,\delta_x)$. Similar to the argument for the second property, it could be shown that
\[T^{i+N}(B(x,\delta_x))\cap Sp_l(X,T)=\varnothing.\]
This then implies that 
\[T^{-1}(T^{i+N}(z))=T^{i-1+N}(z).\]
Note that we also have that $T^{i-1+N}(B(x,\delta_x))\cap Sp_l(X,T)=\varnothing$. Then by a recursive use of this argument, we finally deduce that
\[T^{-i}(T^{i+N}(z))=T^N(z).\]
Denote $\mathcal{U}=\{U_x\}_{x\in X\setminus C}$. It is clear that $\mathcal{U}$ is an open cover of $X\setminus C$. 

Now we define an open cover of $C$ as follows. Recall that
\[C=\bigcup_{0\le i\le 2N-1}T^{-i}(Sp_l(X,T))=\bigcup_{1\le k\le q}\bigcup_{0\le i\le 2N-1}T^{-i}(\{\omega_k\}).\]
For every $1\le k\le q$, let $\tilde{U}_k$ be an open neighborhood of $\omega_k$ such that the sets
\begin{align}\label{tu}
\overline{\tilde{U}_k}, T(\overline{\tilde{U}_k}),\cdots, T^{2N-1}(\overline{\tilde{U}_k})
\end{align}
are pairwise disjoint. For every $x\in T^{-2N+1}(\{\omega_k\})$, let $\tilde{V}_{x,k}$ be a clopen neighborhood of $x$ with 
\[T^{2N-1}(\tilde{V}_{x,k})\subset\tilde{U}_k.\]
Define
\[V_k=\bigcup_{x\in T^{-2N+1}(\{\omega_k\})}T^{-(2N-1)}(T^{2N-1}(\tilde{V}_{x,k})).\]
Then $V_k$'s are clopen sets. We claim that $V_k$ satisfies the following properties:

$\bullet$ $\overline{V_k}, T(\overline{V_k}), \cdots, T^{2N-1}(\overline{V_k})$ are pairwise disjoint: If there are $i,j\in\{0,1,\cdots,2N-1\}$ with $i<j$ such that
\begin{align}\label{eqcrazy}
T^i(\overline{V_k})\cap T^j(\overline{V_k})\ne\varnothing.
\end{align}
This will follows that
\[T^{2N-1}(\overline{V_k})\cap T^{2N-1+j-i}(\overline{V_k})\ne\varnothing.\]
Note that since
\[T^{2N-1}(\overline{V_k})=\bigcup_{T^{2N-1}(x)=w_k}T^{2N-1}(\overline{\tilde{V}_{x,k}})\subset \overline{\tilde{U}_k},\]
one see that $\overline{\tilde{U}_k}\cap T^{j-i}(\overline{\tilde{U}_k})\ne\varnothing$ with $j-i\in\{1,2,\cdots,2N-1\}$. This contradicts \eqref{tu};

$\bullet$ $T^{-1}(T^{i+1}(V_k))=T^i(V_k)$ for $i\in\{0,1,\cdots,2N-2\}$: It suffices to show that
\[T^{-1}(T^{i+1}(V_k))\subset T^i(V_k).\]
Let $x\in T^{-1}(T^{i+1}(V_k))$, that is,
\[T(x)\in T^{i+1}(V_k))=\bigcup_{x\in T^{-2N+1}(\{\omega_k\})}T^{-2N+i+2}(T^{2N-1}(\tilde{V}_{x,k})).\]
Note that the equality follows by Lemma \ref{lmc} since $T$ is surjective. Since $-2N+i+2\le 0$, again by Lemma \ref{lmc}, we have
\[x\in \bigcup_{x\in T^{-2N+1}(\{\omega_k\})}T^{-2N+i+1}(T^{2N-1}(\tilde{V}_{x,k}))=T^i(V_k).\]

$\bullet$ $\mathcal{V}=\{T^i(V_k): 1\le k\le q, 0\le i\le 2N-1\}$ covers $C$: Let $x\in C$ be arbitrary. Then there is $1\le k\le q$ and $0\le j\le 2N-1$ such that
\[x\in T^{-j}(\{\omega_k\}).\]
Find $z\in X$ with $T^{2N-1-j}(z)=x$ since $T$ is surjective. Then one can see that
\[T^{2N-1}(z)=T^{j}(T^{2N-1-j}(z))=T^j(x)=\omega_k,\]
that is, $z\in T^{-2N+1}(\{\omega_k\})\subset \tilde{V}_{z,k}\subset V_k$. Consequently, $x=T^{2N-1-j}(z)\in T^{2N-1-j}(V_k)$. As $2N-1-j\in\{0,1,\cdots,2N-1\}$, the claim then follows.

Now we have defined an open cover of $C$, which follows that
\begin{align*}
&\mathcal{U}\cup\mathcal{V}\\
=\{U_x, T^i(V_k): x\in X\setminus &C, 1\le k\le q, 0\le i\le 2N-1\}
\end{align*}
is an open cover of $X$. Since $X$ is compact, it has a finite subcover $\mathcal{U}_1$. Let 
\[\mathcal{U}_2=\mathcal{U}_1\cap\mathcal{U}\]
Then $\mathcal{U}_2$ is a subcollection of $\mathcal{U}$ such that 
\[\mathcal{U}_2\cup\{T^i(V_k): 1\le k\le q, 0\le i\le 2N-1\}\]
is an open cover of $X$. We note that $\mathcal{U}_2$ might not be an open cover of $X\setminus C$. Now we denote 
\[\mathcal{U}_2=\{U_1, U_2,\cdots,U_s\}.\]
By the previous argument, $U=U_1, V=U_2$ satisfies the condition of Lemma \ref{lm7}, which follows that, by applying Lemma \ref{lm7}, there exists a clopen set $W_2$ with the following properties:

$(P'_1)$ $U_1\subset W_2$;

$(P'_2)$ $U_2\subset\bigcup_{0\le i\le 2N-1}T^{-i}(W_2)$;

$(P'_3)$ $T^{-i}(T^i(W_2))=W_2$ for $i=1,2,\cdots,N-1$;

$(P'_4)$ $T^{-2N+1}(\overline{W_2}),\cdots, T^{-N}(\overline{W_2})$ are pairwise disjoint.

Note that then $U=W_2, V=U_3$ satisfies the condition of Lemma \ref{lm7}. Applying Lemma \ref{lm7} again, we then have a clopen set $W_3$ such that

$(P''_1)$ $W_2\subset W_3$;

$(P''_2)$ $U_3\subset\bigcup_{0\le i\le 2N-1}T^{-i}(W_3)$;

$(P''_3)$ $T^{-i}(T^i(W_3))=W_3$ for $i=1,2,\cdots,N-1$;

$(P''_4)$ $T^{-2N+1}(\overline{W_3}),\cdots, T^{-N}(\overline{W_3})$ are pairwise disjoint.

Note that then $U_1\cup U_2\cup U_3\subset\bigcup_{0\le i\le 2N-1}T^{-i}(W_3)$. Upon a recursive applying of this procedure, we are given a clopen map $W_s$ such that
\begin{align}\label{equ}
\bigcup_{1\le j\le s}U_j\subset\bigcup_{0\le i\le 2N-1}T^{-i}(W_s)
\end{align}
and the sets
\[T^{-2N+1}(\overline{W_s}),\cdots, T^{-N}(\overline{W_s})\]
are pairwise disjoint. Note that this follows that $T^{-N+1}(\overline{W_s}),\cdots,\overline{W_s}$ are also disjoint. To see this, if $x\in T^{-N+i}(\overline{W_s})\cap T^{-N+j}(\overline{W_s})$ for some $1\le i\ne j\le N$, since $T$ is surjective, let $y$ be such that $T^N(y)=x$. Then $T^{2N-i}(y)\in\overline{W_s}$ and $T^{2N-j}(y)\in\overline{W_s}$, which contradicts the fact that $T^{-2N+i}(\overline{W_s})\cap T^{-2N+j}(\overline{W_s})=\varnothing$. 

As what we have mentioned before, 
\[\{U_1, U_2,\cdots,U_s\}\cup\{T^i(V_k): 1\le k\le q, 0\le i\le 2N-1\}\]
is an open cover of $X$, which according to \eqref{equ} follows that
\[\bigcup_{0\le i\le 2N-1}T^{-i}(W_s)\cup \bigcup_{1\le k\le q}\bigcup_{0\le i\le 2N-1}T^i(V_k)=X.\]
Now we are sufficiently ready to define a Rokhlin cover. In fact, by Lemma \ref{lmc} and previous arguments, the following collections of clopen sets are disjoint $N$-towers:

{\bf Tower No. 1:} $\mathcal{T}_1=\{T^{-i}(W_s), i=0,1,\cdots,N-1\}$: We have shown the disjointness of their closures. Since $T$ is surjective, we also have $T(T^{-i-1}(W_s))=T^{-i}(W_s)$;

{\bf Tower No. 2:} $\mathcal{T}_2=\{T^{-i}(W_s), i=N,N+1,\cdots,2N-1\}$: We have also shown the disjointness of their closures. Again, since $T$ is surjective, we also have $T(T^{-i-1}(W_s))=T^{-i}(W_s)$;

{\bf Towers associated with $Sp_l(X,T)$:} For $1\le k\le q$, define
\[\tilde{\mathcal{T}^1_k}=\{T^i(V_k): 0\le i\le N-1\}\ {\rm and}\ \tilde{\mathcal{T}^2_k}=\{T^i(V_k): N\le i\le 2N-1\}.\]
The disjointness of their closures are established by \eqref{eqcrazy}. 

These towers cover the space $X$ in the following manner:
\[\textstyle \left(\bigcup\mathcal{T}_1\right)\cup\left(\bigcup\mathcal{T}_2\right)\cup\left(\bigcup_{1\le k\le q}\bigcup\tilde{\mathcal{T}^1_k}\right)\cup\left(\bigcup_{1\le k\le q}\bigcup\tilde{\mathcal{T}^2_k}\right)=X.\]
Since we have $2q+2=2|Sp_l(X,T)|+2$ disjoint towers, we finally conclude that
\[{\rm dim}_{\rm Rok}(X,T)\le 2|Sp_l(X,T)|+1.\]
The proof is now finished.
\end{proof}

\begin{cor}
Let $X$ be a one-sided subshift with $\sigma(X)=X$. If $X$ is aperiodic and has nonsuperlinear-growth complexity, then
\[{\rm dim}_{\rm Roc}(\tildeX,\sigma_{\tildeX})\le 2\lceil 2d\rceil+1,\]
where $d=\liminf_{n\to\infty}p_X(n)/n$.
\end{cor}
\begin{proof}
This follows from Theorem \ref{zdr}, Theorem \ref{shift}, and Lemma \ref{useful}.
\end{proof}

\section{The dynamic Asymptotic dimension of a local homeomorphism}\label{sec:5}
\subsection{The tower dimension of a local homeomorphism}
\begin{df}\label{td}
Let $X$ be a compact metric space and $T: X\to X$ be a local homeomorphism. The {\it tower dimension} of $(X,T)$ is defined as the smallest nonnegative integer $d\in\N$ satisfying that, for every finite set $E\Subset\Z$, there are finitely many pairs
\[(V_1, S_1), \cdots, (V_s, S_s)\]
with the following properties.

(1) For every $1\le i\le s$, $V_i\subset X$ is an open set and $S_i\Subset\N$ is a finite subset;

(2) For every $1\le i\le s$ and distinct $m,n\in S_i$, $T^{-m}(\overline{V_i})\cap T^{-m}(\overline{V_i})=\varnothing$;

(3) The family $\{T^{-n}(V_i): n\in S_i, 1\le i\le s\}$ has chromatic number at most $d+1$;

(4) The family $\{T^{-n}(V_i): n\in S_i, 1\le i\le s\}$ forms an open cover of $X$;

(5) For every $x\in X$, there is $1\le i\le s$ and $n\in S_i$ such that
\[x\in T^{-n}(V_i)\ \ {\rm and}\ \ E+\{n\}\subset S_i.\]
\end{df}

\begin{lem}\label{td1}
Let $X$ be a compact metric space and $T: X\to X$ be a local homeomorphism. Then
\[{\rm dim}_{\rm tow}(X,T)\le 2{\rm dim}_{\rm Rok}(X,T)+1.\]
\end{lem}
\begin{proof}
It suffices to consider the case of ${\rm dim}_{\rm Rok}(X,T)<\infty$, or else there is nothing to be shown. Assume ${\rm dim}_{\rm Rok}(X,T)=d\in\N$. 

Let $E\Subset \N$ be a nonempty finite subset. We may assume $E\setminus\{0\}\ne\varnothing$, without loss of generality. Define
\[M=1+2\max \{|n|: n\in E\}\ \ {\rm and}\ \ N=2+3\max\{|n|: n\in E\}.\]
Since ${\rm dim}_{\rm Rok}(X,T)=d$, there are at most $d+1$ $N$-Rokhlin towers $\mathcal{T}_i\,(i=0,1,\cdots,d)$, each of which is of the following form by the definition of the Rokhlin dimension:
\[\mathcal{T}_i=\{U^i, T^{-1}(U^i),\cdots, T^{-(N-1)}(U^i)\},\]
where $U^i$ is an open set (serving as the base of $\mathcal{T}_i$) for every $i=0,1,\cdots,d$. Now denote $S=\{0,1,\cdots,N-1\}$ and consider the following pairs
\[(U^0, S), (U^1,S), \cdots, (U^{d}, S), (T^{-M}(U^0), S), \cdots, (T^{-M}(U^{d}), S).\]
We now show that the above pairs satisfy the properties (1)--(5) in Definition \ref{td}.

(1) Since $T$ is continuous, the first property trivially holds;

(2) Let $m,n\in S$ be distinct elements and $0\le i\le d$. That $T^{-m}(\overline{U^i})\cap T^{-n}(\overline{U^i} )=\varnothing$ follows directly from the definition of the Rokhlin towers. Note that this also implies that $T^{-m}(\overline{T^{-M}(U^i)})\cap T^{-n}(\overline{T^{-M}(U^i)})=\varnothing$, since by Lemma \ref{lmc}
\begin{align*}
T^{-m}(\overline{T^{-M}(U^i)})\cap T^{-n}(\overline{T^{-M}(U^i)})&=T^{-m}(T^{-M}(\overline{U^i}))\cap T^{-n}(T^{-M}(\overline{U^i}))\\
&=T^{-m-M}(\overline{U^i})\cap T^{-n-M}(\overline{U^i})\\
&=T^{-M}(T^{-m}(\overline{U^i})\cap T^{-n}(\overline{U^i}))=\varnothing.
\end{align*}

(3) As there are $2d+2$ towers each of which is a disjoint collection of open sets, its chromatic number is at most $2d+2$;

(4) This is immediate since $\{T^{-n}(U^i): n\in S, 0\le i\le d\}$ has already formed an open cover of $X$;

(5) Let $x\in X$. Since 
\[\{T^{-n}(U^i): n\in S, 0\le i\le d\}\]
is an open cover of $X$, there are $0\le i\le d$ and $n\in S=\{0,1,\cdots,N-1\}$ such that
\[x\in T^{-n}(U^i).\]

$\bullet$ if $0\le n\le M-1$, then for every $t\in E$, we have
\[0\le t+n\le M-1+\max E\le 3\max\{|n|: n\in E\}=N-2\le N-1,\]
which follows that $E+\{n\}\subset S$;

$\bullet$ If $M\le n\le N-1$, then $n-M\ge0$, which implies that
\[x\in T^{-n}(U^i)=T^{-(n-M)}(T^{-M}(U^i)).\]
Note that we also have $n-M\in S$ since $0\le n-M\le N-1$. Then this follows that, for every $t\in E$,
\begin{align*}
0\le t+n-M&\le\max E+N-1-M\\
&\le\max\{|n|: n\in E\}+N-1-(1+2\max\{|n|: n\in E\})\\
&=N-2-\max\{|n|: n\in E\}\le N-1,
\end{align*}
which again implies that $E+\{n-M\}\subset S$.

Finally, by the chromatic $\le 2d+2$ in (3), we conclude that ${\rm dim}_{\rm tow}(X,T)\le 2d+1=2{\rm dim}_{\rm Rok}(X,T)+1$. The proof is now completed.
\end{proof}

\begin{cor}\label{cor}
Let $X$ be a zero-dimensional compact metric space and $T: X\to X$ an aperiodic surjective local homeomorphism. Then
\[{\rm dim}_{\rm tow}(X,T)\le 4|Sp_l(X,T)|+3.\]
\end{cor}

\begin{proof}
It suffices to assume $|Sp_l(X,T)|<\infty$. Then by Theorem \ref{zdr}, 
\[{\rm dim}_{\rm Rok}(X,T)\le 2|Sp_l(X,T)|+1.\]
Applying Lemma \ref{td1}, we conclude that
\[{\rm dim}_{\rm tow}(X,T)\le 2{\rm dim}_{\rm Rok}(X,T)+1\le 4|Sp_l(X,T)|+3.\]
The corollary then follows.
\end{proof}

\begin{cor}\label{corr}
Let $X$ be a one-sided shift space with $\sigma(X)=X$. If $X$ is aperiodic and has nonsuperlinear-growth complexity, then
\[{\rm dim}_{\rm tow}(\tildeX,\sigma_{\tildeX})\le 4\lceil 2d\rceil+3,\]
where $d=\liminf_{n\to\infty}p_X(n)/n$.
\end{cor}
\begin{proof}
This follows from Corollary \ref{cor}, Theorem \ref{shift}, and Lemma \ref{useful}.
\end{proof}

\subsection{The amenability dimension of a local homeomorphism}

\begin{df}
Let $X$ be a compact metric space and $T: X\to X$ a local homeomorphism. Let $Y$ be a metric space and $\af: \Z\curvearrowright Y$ an action on $Y$ by isometries. 

Let $E\Subset \Z$ be a finite subset and $\varepsilon>0$ a positive real number. A map $\varphi: X\to Y$ is said to be {\it $(E,\varepsilon)$-equivariant}, if
\[d_Y(\varphi(y), \af_n(\varphi(x)))<\varepsilon\]
for all $x,y\in X$ and $n\in E$ with $y\in T^n(\{x\})$.
\end{df}

\begin{df}\label{amd}
Let $X$ be a compact metric space, and $T:X\to X$ a local homeomorphism. The {\it amenability dimension} ${\rm dim}_{\rm am}(X,T)$ of $(X,T)$ is the smallest integer $d\in\N$ with the property that, for every finite subset $E\Subset\Z$ and $\varepsilon>0$, there is an $(E,\varepsilon)$-equivariant continuous map 
\[\varphi: X\to P_d(\Z).\]
\end{df}

\begin{lem}\label{fad}
For every topological dynamical system $(X,T)$ where $X$ is a compact metric space and $T: X\to X$ a surjective local homeomorphism, if $Sp_l(X,T)$ is a finite set consisting of isolated points, then
\[{\rm dim}_{\rm am}(X,T)\le {\rm dim}_{\rm tow}(X,T).\]
\end{lem}

\begin{proof}
We may assume that ${\rm dim}_{\rm tow}(X,T)<\infty$, or else there will be nothing to prove. Write $d={\rm dim}_{\rm tow}(X,T)\in\N$. Fix a finite subset $E\Subset\Z$ and a positive number $\varepsilon>0$. Without loss of generality, by enlarging $E$, we may assume $E=E^{-1}$ and $0\in E$.

Let $N\in\N$ so that $(d+1)(d+2)/N<\varepsilon$. By the definition of the tower dimension, there are finitely many pairs
\[(V_1, S_1), \cdots, (V_s, S_s)\]
with the following properties.

($P_1$) For every $1\le i\le s$, $V_i\subset X$ is an open set and $S_i\Subset\N$ is a finite subset;

($P_2$) For every $1\le i\le s$ and distinct $m,n\in S_i$, $T^{-m}(\overline{V_i})\cap T^{-m}(\overline{V_i})=\varnothing$;

($P_3$) The family $\{T^{-n}(V_i): n\in S_i, 1\le i\le s\}$ has chromatic number at most $d+1$;

($P_4$) The family $\{T^{-n}(V_i): n\in S_i, 1\le i\le s\}$ forms an open cover of $X$;

($P_5$) For every $x\in X$, there is $1\le i\le s$ and $n\in S_i$ such that
\[x\in T^{-n}(V_i)\ \ {\rm and}\ \ \{n\}+\sum_{N\ {\rm copies}}E\subset S_i,\]
where $\sum_{N\ {\rm copies}}E=\{n_1+n_2+\cdots+n_N: n_i\in E, 1\le i\le N\}$ is clearly a finite set. We will simply denote this finite subset by $\sum_NE$. By $(P_2), (P_5)$ and the fact that $X$ is compact, we can take a $\delta>0$ such that, for every $x\in X$, there exist an index $1\le i\le s$, an $n\in S_i$ with
\[d(x, X\setminus T^{-n}(V_i))>\delta\]
and $\{n\}+\sum_NE\subset S_i$. For every $1\le i\le s$ and $n\in S_i$ define the continuous function $\hat{g}_{i,n}$ on $X$ by
\[\textstyle\hat{g}_{i,n}(x)=\min\{1,\frac{1}{\delta}d(x,X\setminus T^{-n}(V_i))\}\]
and define $g_i: X\to [0,1]$ for every $1\le i\le s$ by
\begin{align*}
g_i(x)=
\begin{cases}
\max\{\hat{g}_{i,n}(y): n\in S_i, y=T^{-n}(x)\}, \ \ &{\rm if}\ x\in X\setminus\bigcup_{n\in S_i}Sp_l(X,T^n),\\
\max\{\hat{g}_{i,n}(y): n\in S_i, y\in X, T^n(y)=x\}, &{\rm if}\ x\in\bigcup_{n\in S_i}Sp_l(X,T^n).
\end{cases}
\end{align*}
We note that such $g_i$'s are all continuous. First, for every $n\in\N$, one sees that
\[Sp_l(X,T^n)\subset \bigcup_{0\le k\le n-1}T^k(Sp_l(X,T)),\]
which is hence a finite set consisting of isolated points since $T$ is a local homeomorphism. This then follows that, as a finite union of finite clopen set, $\bigcup_{n\in S_i}Sp_l(X,T^n)$ is a finite clopen set. Note that $g_i$'s are then continuous on the clopen set $X\setminus \bigcup_{n\in S_i}Sp_l(X,T^n)$, because $T^n(y)\mapsto y$ is continuous on $X\setminus \bigcup_{n\in S_i}Sp_l(X,T^n)$. This verifies that $g_i$'s are continuous functions defined on the whole space $X$, taking values in $[0,1]$.

We also note that for every $n\in S_i$, the support of the function $g_i\circ T^n$ is contained in $T^{-n}(V_i)$. To see this, let $x\notin T^{-n}(V_i)$. By the definition of $g_i$'s, we know that
\[g_i\circ T^n(x)=\max\{\hat{g}_{i,m}(y): m\in S_i, y\in T^{-m}(T^n(x))\}.\]
Since $x\notin T^{-n}(V_i)$, $T^n(x)\notin V_i$. Then for every $y\in T^{-m}(T^n(x))$, we have $y\in T^{-m}(V_i^c)=T^{-m}(V_i)^c$, that is, $y\notin T^{-m}(V_i)$. This then immediately follows that,
\[d(y,X\setminus T^{-m}(V_i))=0,\]
whence $\hat{g}_{i,m}(y)=0$, i.e., $g_i\circ T^n(x)=0$. This shows that ${\rm supp}(g_i\circ T^n)\subset T^{-n}(V_i)$.

Let $1\le i\le s$. Set 
\[B_{i,N}=\bigcap_{m\in \sum_NE}(\{m\}+S_i)\ \ {\rm and}\ \ B_{i,0}=\Z\setminus \bigcap_{m\in E}(\{m\}+S_i).\]
For $k=1,2,\cdots,N-1$, define
\[B_{i,k}=\left(\bigcap_{m\in\sum_{k}E}(\{m\}+S_i)\right)\setminus \left(\bigcap_{m\in\sum_{k+1}E}(\{m\}+S_i)\right).\]
Similarly to the arguments in Theorem 5.2 of \cite{K}, the finite sets $B_{i,k}$ for $k=0,1,\cdots,N$ form a partition of $\Z$, such that for every $m\in E$, they satisfy the following properties:

(1) $\{m\}+B_{i,0}\subset B_{i,0}\cup B_{i,1}$;

(2) $\{m\}+B_{i,k}\subset B_{i,k-1}\cup B_{i,k}\cup B_{i,k+1}$ for $k=1,2,\cdots,N-1$;

(3) $\{m\}+B_{i,N}\subset B_{i,N-1}\cup B_{i,N}$.

\noindent For every $m\in\Z$, take $k$ such that $m\in B_{i,k}$. Then define the function
\begin{align*}
\hat{h}_{i,m}(x)=
\begin{cases}
\frac{k}{N}\cdot g_i\circ T^m(x), \ \ &{\rm if}\ m\ge0,\\
\frac{k}{N}\cdot g_i\circ T^m(x), &{\rm if}\ m<0\ {\rm and}\ x\in X\setminus\bigcup_{n\in B_{i,k}\cap\N_-}Sp_l(X,T^{-n})\\
\frac{k}{N}\cdot \max\{g_i(y): y\in T^m(x)\}, &{\rm if}\ m<0\ {\rm and}\ x\in \bigcup_{n\in B_{i,k}\cap\N_-}Sp_l(X,T^{-n}).
\end{cases}
\end{align*}
Upon using the same argument showing the continuity of $g_i$'s, we could claim that $\hat{h}_{i,m}\in C(X)$, and also it is straightforward to verify that $|\hat{h}_{i,m}(y)-\hat{h}_{i,m-n}(x)|\le 1/N$ for all $x\in X$, $y\in T^n(\{x\})$ and $n\in E$. Set
\[H=\sum_{1\le i\le s}\sum_{m\in\Z}\hat{h}_{i,m}.\]
Note that $H$ is well-defined. Since for every $x\in X$, there is $1\le i\le s$ and $n\in S_i$ such that $d(x,X\setminus T^{-n}(V_i))>\delta$, and $\{n\}+\sum_NE\subset S_i$, then $n\in B_{i,N}$, which follows that
\[\hat{h}_{i,n}(x)\ge \hat{g}_{i,n}(x)=1.\]
This then implies that $H\ge1$.  Let $1\le i\le s$ and $n\in\Z$, let
\[h_{i,n}=\hat{h}_{i,n}/H,\]
and $\varphi: X\to P_d(\Z)$ defined by
\[\varphi(x)(m)=\sum_{1\le i\le s}h_{i,m}(x).\]
for $x\in X$ and $m\in\Z$. One observes that $\varphi$ is continuous since so is each $h_{i,m}$.

Now it only suffices to show that $\varphi$ is $(E,\varepsilon)$-equivariant. On one hand, by the definition of $h_{i,m}$ and $\hat{h}_{i,m}$, and using the same argument of the last paragraph of Theorem 5.2 in \cite{K}, one calculates that
\begin{align*}
|h_{i,m}(y)-h_{i,m-n}(x)|&\le (1/H(y))|\hat{h}_{i,m}(y)-\hat{h}_{i,m-n}(x)|+|(1/H(y))-(1/H(x))|\hat{h}_{i,m-n}(x)\\
&\le (d+2)/N
\end{align*}
for every $x\in X$, $n\in E$, $m\in\Z$ and $y\in T^n(\{x\})$.  On the other hand, also note that since for every $x\in X$, the set of all $i\in\{1,2,\cdots,s\}$ such that $x\in \bigcup_{n\in S_i}T^{-n}(V_i)$ has cardinality at most $d+1$, which implies that, for every $x\in X$, $n\in E$ and $y\in T^n(\{x\})$, we have
\begin{align*}
\rho(\varphi(y),\af_n(\varphi(x)))&=\sum_{m\in\Z}|\varphi(y)(m)-\varphi(x)(m-n)|\\
&\le \sum_{m\in\Z}\sum_{1\le i\le s}|h_{i,m}(y)-h_{i,m-n}(x)|\\
&\le (d+1)(d+2)/N<\varepsilon.
\end{align*}
The Lemma then follows.
\end{proof}

\begin{cor}
Let $X$ be a one-sided shift space with $\sigma(X)=X$. If $X$ is aperiodic and has nonsuperlinear-growth complexity, then
\[{\rm dim}_{\rm am}(\tildeX,\sigma_{\tildeX})\le 4\lceil 2d\rceil+3,\]
where $d=\liminf_{n\to\infty}p_X(n)/n$.
\end{cor}
\begin{proof}
This follows from Theorem \ref{shift}, Lemma \ref{fad} and Corollary \ref{corr}.
\end{proof}

\subsection{The dynamic asymptotic dimension of a local homeomorphism}

\begin{df}[\cite{GWY}, Definition 5.1]\label{dad}
Let $\mathcal{G}$ be an $\acute{\rm e}$tale groupoid. Then we say $\mathcal{G}$ has {\it dynamic asymptotic dimension} $d\in\N$, and write ${\rm dad}(\mathcal{G})=d$, if $d$ is the smallest number with the following property.

For every open relatively compact subset $K$ of $\mathcal{G}$ there are open subsets $U_0, U_1,\cdots,U_d$ of $\mathcal{G}^0$ that covers $s(K)\cup r(K)$ such that for each $i$, the set 
\[\{g\in K: s(g),r(g)\in U_i\}\]
is contained in a relatively compact subgroupoid of $\mathcal{G}$.
\end{df}

\begin{lem}\label{lmeq}
Let $X$ be a compact metric space, $T: X\to X$ a surjective local homeomorphism. Let $\varphi: X\to P_d(\Z)$ be an $(E,\varepsilon)$-equivariant map, where the positive number $\varepsilon$, the natural number $d\in\N$ and the finite subset $E\Subset \Z$ are given. Then there exists a finite subset $S\subset\Z$ and an $(E,\varepsilon)$-equivariant map $\varphi':X\to P_d(\Z)$ such that 
\[\varphi'(X)\subset P(S)\cap P_d(\Z).\]
\end{lem}

\begin{proof}
Since for any $\delta>0$, the family
\[\{B_{\delta}(P(S)\cap P_d(\Z)): S\Subset\Z\ {\rm a\ finite\ set}\}\]
is an open cover of $P_d(\Z)$, which follows that for every compact set $K\subset P_d(\Z)$, and any $\delta>0$, there is a finite subset $S\subset \Z$ with $K\subset B_\delta(P(S)\cap P_d(\Z))$. Let
\[\delta=\min\left\{1,\frac{2}{\varepsilon}\min\{\varepsilon-\sup\{\rho(\varphi(y),\af_n\varphi(x)): y\in T^n(\{x\})\}:n\in E\}\right\}>0.\]
Let $S$ be a finite subset with 
\[\varphi(X)\subset B_{\delta/2}(P(S)\cap P_d(\Z)).\]
Write $\varphi(x)=\sum_{n\in\Z}t_n(x)\cdot\mu_n$, where $\mu_n$ is the Dirac measure on $n$ and $t_n: X\to [0,1]$ is a continuous function. One then sees that for any $\mu=\sum_{n\in S}s_n\mu_n\in P(S)$, 
\[\rho(\mu,\varphi(x))=\sum_{n\in S}|s_n-t_n(x)|+\sum_{n\notin S}t_n(x)\ge \sum_{n\notin S}t_n(x).\]
Upon taking the infimum over all $\mu\in P(S)$, and by the fact that $\varphi(X)\subset B_{\delta/2}(P(S)\cap P_d(\Z))$, we have
\[\sum_{n\notin S}t_n(x)<\delta/2.\]
Define
\[\varphi'(x)=\sum_{n\in S}\frac{t_n(x)}{\sum_{n\in S}t_n(x)}\mu_n,\]
where $\sum_{n\in S}t_n(x)>1-\delta/2$. We then see that $\varphi': X\to P(S)\cap P_d(\Z)$ is a continuous function, and for any $x\in X$,
\[\rho(\varphi(x), \varphi'(x))=\sum_{n\in S}t_n(x)\left(\frac{1}{\sum_{n\in S}t_n(x)}-1\right)+\sum_{n\notin S}t_n(x)=2\left(1-\sum_{n\in S}t_n(x)\right)<\delta.\]
This follows that, for any $n\in E$, $x,y\in X$ with $y\in T^n(\{x\})$,  we have
\begin{align*}
\rho(\varphi'(y),\af_n(\varphi'(x)))&=\sum_{m\in \Z}\frac{|t_m(y)-t_{m+n}(x)|}{\sum_{k\in S} t_k(y)}\\
&\le \frac{\varepsilon-(\varepsilon-\rho(\varphi(y),\af_n(\varphi(x))))}{1-\delta/2}\\
&\le \frac{\varepsilon-\delta\varepsilon/2}{1-\delta/2}=\varepsilon,
\end{align*}
which verifies that $\varphi'$ is $(E,\varepsilon)$-equivariant.
\end{proof}

\begin{lem}[Lemma 4.1, \cite{GWY}]\label{4.1}
Let $d\in\N$ and $\af: \Z\curvearrowright P_d(\Z)$ be the simplicial action. Let $\delta>0$.

For each $i\in\{0,1,\cdots,d\}$, define
\[V_i=B_{1/(3\cdot10^i)}(P_{i}(\Z))\setminus \overline{B_{5/(2\cdot10^i)}(P_{i-1}(\Z))}.\]
Then the collection $\{V_0,V_1,\cdots,V_d\}$ is an open cover of $P_d(\Z)$, consisting of $\Z$-invariant subsets. 

Moreover, for each $i\in\{0,1,\cdots,d\}$ and $i$-simplex $\Delta$, define
\[V_{i,\Delta}=B_{1/(3\cdot10^i)}(\Delta)\setminus \overline{B_{5/(2\cdot10^i)}(P_{i-1}(\Z))}.\]
Then 
\[V_i=\bigcup_{\Delta\ {\rm an}\ i-{\rm simplex}}V_{i,\Delta},\]
the action $\af$ permutes the distinct $V_{i,\Delta}$, and for any two distinct $i$-simplices $\Delta$ and $\Delta'$, we have
\begin{align}\label{inq}
{\rm dist}(V_{i,\Delta}, V_{i,\Delta'})\ge \frac{1}{3\cdot10^i}\ge \frac{1}{3\cdot10^d}.
\end{align}
\end{lem}

\begin{lem}\label{fdad}
Let $X$ be a compact metric space and $T: X\to X$ be a surjective local homeomorphism. If $Sp_l(X,T)$ is a finite subset of isolated points of $X$, then
\[{\rm dad}(\mathcal{G}_{(X,T)})\le {\rm dim}_{\rm am}(X,T).\]
\end{lem}

\begin{proof}
We may assume that ${\rm dim}_{\rm am}(X,T)<\infty$, or else there will be nothing to be shown. Denote $d={\rm dim}_{\rm am}(X,T)$ and $\mathcal{G}=\mathcal{G}_{(X,T)}$ for abbreviation.

Let $K\subset \mathcal{G}\subset X\times\Z\times X$ be an open relatively compact subset arbitrarily. It is then clear that
\[K\subset X\times E\times X\]
for some finite subset $E\Subset\Z$. Without loss of generality, we may assume that $K$ is of this form, by enlarging $K$.

Write $Sp_l(X,T)=\{\omega_1,\cdots,\omega_q\}$ for $q=|Sp_l(X,T)|$. Let $\varepsilon=\frac{1}{3\cdot10^d}$. Then there is a continuous map $\varphi: X\to P_d(\Z)$ such that
\begin{align}\label{eqv}
\textstyle\sup\{\rho(\varphi(y),\af_{n}(\varphi(x))): y\in T^{n}(\{x\}), n\in E\}<\frac{1}{3\cdot10^d}.
\end{align}

By Lemma \ref{lmeq}, we may assume $\varphi(X)\subset P_d(\Z)\cap P(S)$ for some finite set $S\Subset\Z$. Define 
\[F=\{n\in\Z: (n+S)\cap S\ne\varnothing\}.\]
It is clear that $F$ is a finite subset of $\Z$, as $S$ is finite.

Now for every $i\in\{0,1,\cdots,d\}$ and every $i$-simplex $\Delta$ in $P_d(\Z)$, let 
\[V_i\subset P_d(\Z)\ \ {\rm and}\ \ V_{i,\Delta}\subset V_i\]
be as in Lemma \ref{4.1}, and define
\[\tilde{U}_i=\varphi^{-1}(V_i)\ \ {\rm and}\ \ \tilde{U}_{i,\Delta}=\varphi^{-1}(V_{i,\Delta}).\]
It is then immediate that $\{\tilde{U}_0,\tilde{U}_1,\cdots,\tilde{U}_d\}$ is an open cover of $X$, and that every $\tilde{U}_i$ is a disjoint union of $\tilde{U}_{i,\Delta}$'s as $\Delta$ is taken over all $i$-simplices. We claim that, for any $i\in\{0,\cdots,d\}$ and any $i$-simplex $\Delta$, if

(i) $x\in \tilde{U}_{i,\Delta}$;

(ii) $n\in E$;

(iii) $T^n(\{x\})\subset\tilde{U}_i$, 

\noindent then $T^n(\{x\})\subset \tilde{U}_{i,\af_n(\Delta)}$.
In fact, let $y\in T^n(\{x\})$ be an arbitrary point. Since $T^n(\{x\})\subset \tilde{U}_i$ and $\tilde{U}_i$ is a disjoint union of $\tilde{U}_{i,\Delta}$'s,  $y\in \tilde{U}_{i,\Delta'}$ for some $i$-simplex $\Delta'$. By \eqref{eqv}, we know that
\[\|\varphi(y)-\af_n(\varphi(x))\|_1<1/(3\cdot10^d).\]
Note that since $x\in \tilde{U}_{i,\Delta}=\varphi^{-1}(V_{i,\Delta})$, $\af_n(\varphi(x))\in\af_n(V_{i,\Delta})=V_{i,\af_n(\Delta)}$ and $\varphi(y)\in V_{i,\Delta'}$. Then the inequality \eqref{inq} yields that $\Delta'=\af_n(\Delta)$, whence $y\in \tilde{U}_{i,\af_n(\Delta)}$.

To complete the proof, for $i=0,1,\cdots,d$, define
\[\textstyle U_i=\{(x,0,x)\in \mathcal{G}^{(0)}: x\in \tilde{U}_i\}\cup\bigcup_{1\le k\le q}{\rm Orb}(\omega_k).\]
Now it is only remained to show that $U_0,U_1,\cdots,U_d$ satisfy the properties in Definition \ref{dad}.

$\bullet$ The $U_i$'s are open subsets of $\mathcal{G}^{(0)}$: Every $U_i$ is a union of two parts: $\{(x,0,x): x\in\tilde{U}_i\}$ and $\bigcup_{1\le k\le q}{\rm Orb}(\omega_k)$.  The first part is open because every $\tilde{U}_i$ is open; the second part is open as well, since every $\omega_k$ is an isolated point, and, $T$ is an open continuous map on a compact metric space, as a local homeomorphism.

$\bullet$ The $U_i$'s cover $r(K)\cup s(K)$: This is immediate since $\tilde{U}_i$'s have already covered $X$.

$\bullet$ Now it suffices to shown that for each $i\in\{0,1,\cdots,d\}$, the set
\[\{g\in K: s(g), r(g)\in U_i\}\]
is contained in a relatively compact subgroupoid of $\mathcal{G}$: Let $g=(x,m-n,y)\in K$ for some $x\in X$, $m,n\in\N$ with $T^m(x)=T^n(y)$, $m-n\in E$ and
\[s(g)=(y,0,y)\in U_i\ \ {\rm and}\ \ r(g)=(x,0,x)\in U_i.\]
We divide this into the following two cases:

(1) If $x,y\notin {\rm Orb}(\omega_k)$ for any $1\le k\le q$. Note that this follows $x,y\in \tilde{U}_i$.  Without loss of generality, we may assume $m\ge n$. This then implies that $T^{m-n}(x)=y$. Let $x\in\tilde{U}_{i,\Delta}$ for some $i$-simplex $\Delta$. By applying the above claim to $x$ and $m-n$, we have $y=T^{m-n}(x)\in\tilde{U}_{i,\af_{m-n}(\Delta)}$. Note that this implies 
\[\varphi(x)\in V_{i,\Delta}\ \ {\rm and}\ \ \varphi(T^{m-n}(x))\in V_{i, \af_{m-n}(\Delta)}.\]
Since both $\varphi(T^{m-n}(x))$ and $\varphi(x)$ are supported on $S$, this forces $((m-n)+S)\cap S\ne\varnothing$. This means that $m-n\in F$, which is a finite set. Therefore, $g$ is contained in the subgroupoid $\mathcal{G}_1$ of $\mathcal{G}$ generated by $X\times F\times X$, which is relatively compact;

(2) If either $x$ or $y$ is in ${\rm Orb}(\omega_k)$ for some $1\le k\le q$. Note that this is equivalent to say that both $x$ and $y$ are in ${\rm Orb}(\omega_k)$, since $T^m(x)=T^n(y)$. As $\bigcup_{1\le k\le q}{\rm Orb}(\omega_k)$ is a discrete subset of $\mathcal{G}$ and $K$ is relatively compact, the set
\[\textstyle\{g=(x,m-n,y)\in K: x,y\in \bigcup_{1\le k\le q}{\rm Orb}(\omega_k), s(g), r(g)\in U_i\}\]
is clearly finite. Therefore, it is contained in a subgroupoid $\mathcal{G}_2$ of $\mathcal{G}$, which is also relatively compact.

Finally, let $\tilde{G}$ be the subgroupoid of $\mathcal{G}$ generated by $\mathcal{G}_1$ and $\mathcal{G}_2$. The above argument concludes that $\{g\in K: s(g), r(g)\in U_i\}$ is contained in $\mathcal{G}$, which is relatively compact, since so are $\mathcal{G}_1$ and $\mathcal{G}_2$. This completes the proof.
\end{proof}

\subsection{The finiteness of dynamic asymptotic dimension}
\begin{thm}\label{mt2}
Let $X$ be a compact metric space, $T: X\to X$ a  surjective local homeomorphism. If $Sp_l(X,T)$ is finite and consists of isolated points in $X$, then
\[{\rm dad}(\mathcal{G}_{(X,T)})\le 2{\rm dim}_{\rm Rok}(X,T)+1.\]
\end{thm}
\begin{proof}
We may assume ${\rm dim}_{\rm Rok}(X,T)<\infty$, or else there will be nothing left to proof. By Lemma \ref{td1}(or Corollary \ref{cor}), $(X,T)$ has finite tower dimension bounded by
\[{\rm dim}_{\rm tow}(X,T)\le 2{\rm dim}_{\rm Rok}(X,T)+1.\]
Further applying Lemma \ref{fad}, one sees that $(X,T)$ has finite amenability dimension estimated by
\[{\rm dim}_{\rm am}(X,T)\le {\rm dim}_{\rm tow}(X,T)\le 2{\rm dim}_{\rm Rok}(X,T)+1.\]
Finally, according to Lemma \ref{fdad}, $\mathcal{G}_{(X,T)}$ has finite dynamic asymptotic dimension with
\[{\rm dad}(\mathcal{G}_{(X,T)})\le {\rm dim}_{\rm am}(X,T)\le 2{\rm dim}_{\rm Rok}(X,T)+1.\]
The Theorem then follows.
\end{proof}

\begin{cor}\label{mt}
Let $X$ be a zero-dimensional compact metric space, $T: X\to X$ an aperiodic surjective local homeomorphism.

If $Sp_l(X,T)$ is finite and consists of isolated points in $X$, then
\[{\rm dad}(\mathcal{G}_{(X,T)})\le 4|Sp_l(X,T)|+3.\]
\end{cor}
\begin{proof}
By Theorem \ref{zdr}, $(X,T)$ has finite Rokhlin dimension with 
\[{\rm dim}_{\rm Rok}(X,T)\le 2|Sp_l(X,T)|+1.\]
Then combining with Theorem \ref{mt2} yields the conclusion.
\end{proof}

\begin{cor}
Let $X$ be a one-sided shift space with $\sigma(X)=X$. If $X$ is aperiodic and has nonsuperlinear-growth complexity, then
\[{\rm dad}(\mathcal{G}_{(\tildeX,\sigma_{\tildeX})})\le 4\lceil 2d\rceil+3,\]
where $d=\liminf_{n\to\infty}p_X(n)/n$.
\end{cor}
\begin{proof}
This follows from Theorem \ref{shift}, Lemma \ref{useful} and Corollary \ref{mt}.
\end{proof}

\vspace{1cm}

\noindent Sihan Wei, School of Mathematics and Statistics, university of Glasgow, Glasgow, UK

{\em Email address}: {\bf sihan.wei@glasgow.ac.uk}
\quad\par
\quad\par
\noindent Zhuofeng He, BIMSA, Beijing, China

{\em Email address}: {\bf zhuofenghe@bimsa.cn}

\end{document}